\documentclass[10pt]{amsart}
\usepackage{amssymb,a4wide}
\usepackage{amsthm}
\usepackage{graphicx}
\usepackage[T2A]{fontenc}
\usepackage[latin1]{inputenc}

\newcommand{\weg}[1]{}
\newcommand{\spann}{\mathrm{span}}
\newcommand{\trace}{\mathrm{trace}}
\newcommand{\grad}{\mathrm{grad}}
\newcommand{\Id}{\mathrm{Id}}

\theoremstyle{plain}
\newtheorem{thm}{Theorem}
\newtheorem*{thm*}{Theorem}
\newtheorem{lem}{Lemma}
\newtheorem{prop}{Proposition}
\newtheorem{cor}{Corollary}
\newtheorem*{con}{Convention}

\theoremstyle{definition}
\newtheorem{defn}{Definition}

\newtheorem{exmp}{Example}
\theoremstyle{remark}
\newtheorem{rem}{Remark}

\title[Proof of the Yano-Obata Conjecture for $h$-projective transformations]{Proof of the Yano-Obata Conjecture for holomorph-projective transformations}
  \author{Vladimir S. Matveev  and  Stefan Rosemann}
\thanks{Institute of Mathematics, FSU Jena, 07737 Jena Germany,\\  vladimir.matveev@uni-jena.de, stefan.rosemann@uni-jena.de}
\thanks{partially supported by  GK 1523 of DFG}
\begin{document}

\begin{abstract}
We prove  the classical Yano-Obata conjecture by showing that the connected component of the group of holomorph-projective transformations of  a closed, connected Riemannian K\"ahler manifold  consists of isometries unless the metric has constant positive holomorphic curvature. 
\end{abstract}

\maketitle

\section{Introduction}
\label{sec:hpr}
\subsection{Definitions and main result.}
Let $(M,g,J)$ be a Riemannian K\"ahler manifold of real dimension $2n\ge 4$. We denote by $\nabla$ the Levi-Civita connection of  $g$. All objects we consider are assumed to be sufficiently  smooth. 
\begin{defn}
\label{def:hplanar}
A regular curve $\gamma:I\to M$ is called \emph{$h$-planar},  if there exist functions $\alpha,\beta:I\rightarrow\mathbb{R}$ such that the ODE 
\begin{equation}\label{eq:eq1} 
\nabla_{\dot \gamma(t)}\dot \gamma(t)=\alpha\dot \gamma(t)+\beta J(\dot \gamma(t))  
\end{equation} 
holds for all $t$, where $\dot \gamma =\tfrac{d}{dt} \gamma$.
\end{defn}
In certain papers, $h$-planar curves are called complex geodesics. The reason is that 
if we view the action of $J$ on the tangent space as the  multiplication with the imaginary unit $i$, 
the property of a curve $\gamma$ to be  $h$-planar  means that  $\nabla_{\dot \gamma(t) }\dot\gamma(t)$
is  proportional to 
$\dot \gamma(t)$  with a complex coefficient of the proportionality $\alpha(t) + i\cdot \beta(t)$. Recall that geodesics (in an arbitrary, not necessary arc length  parameter $t$) 
 of a metric  can be defined as curves   satisfying  the equation $\nabla_{\dot\gamma(t)} \dot\gamma(t) = \alpha(t) \dot \gamma(t)$. 

\begin{exmp} 
\label{ex:ex3}  Consider the 
 complex projective space
\begin{align}
\mathbb{C}P (n)=\{1\mbox{-dimensional complex subspaces of }\mathbb{C}^{n+1}\}\nonumber
\end{align}
with the standard complex structure $J$ and the standard  Fubini-Study metric $g_{FS}$.  Then, a regular curve  $\gamma$ is $h$-planar, if and only if it lies in a projective line. 

Indeed, it is well known that every  projective line $L$ is a totally geodesic submanifold of real 
 dimension two such that its tangent space is invariant with respect to $J$.  
 Since $L$ is totally geodesic, for every regular curve $\gamma:I\to L\subseteq  \mathbb{C}P (n)$ we have 
  $\nabla_{\dot \gamma(t)}\dot \gamma(t)\in T_{\gamma(t)}L$. Since $L$ is two-dimensional, the vectors  $\dot\gamma(t), J(\dot\gamma(t))$ form a basis in $T_{\gamma(t)}L$.  Hence, $\nabla_{\dot \gamma(t)}\dot \gamma(t)=\alpha(t) \dot \gamma(t)+  \beta(t) J(\dot \gamma(t))$  for certain $\alpha(t), \beta(t)$ as we claimed.

 Conversely, given a  regular curve $\sigma$ in $\mathbb{C}P(n)$ that satisfies equation \eqref{eq:eq1} for some functions $\alpha$ and $\beta$, we consider the  projective line $L$ such that $\sigma(0)\in L$ and $\dot \sigma(0)\in T_{\sigma(0)}L$.  Solving the initial value problem $\gamma(0)=\sigma(0)$ and $\dot \gamma(0)=\dot\sigma(0)$ for ODE  \eqref{eq:eq1} with these  functions $\alpha$ and $\beta$ on $(L, g_{FS|L}, J_{|L})$, we find a curve $\gamma$ in $L$. Since $L$ is totally geodesic, 
 this curve satisfies equation \eqref{eq:eq1} on $(\mathbb{C}P(n),g_{FS},J)$. The uniqueness of a solution of an ODE implies that $\sigma$ coincides with $\gamma$ and,  hence, is contained in $L$.
\end{exmp}

\begin{defn}
\label{defn:hpro}
Let  $g$ and $\bar{g}$  be Riemannian metrics on $M$ such that they 
are Kähler with respect to the same complex structure $J$. They are  called \textit{$h$-projectively equivalent}, if every  $h$-planar curve of $g$ is an $h$-planar curve of $\bar{g}$ and vice versa. 
\end{defn}

\begin{rem}
\label{rem:affishpro}
If two K\"ahler metrics $g$ and $\bar{g}$ on $(M,J)$ are \emph{affinely equivalent} (i.e., if their Levi-Civita connections $\nabla$ and $\bar{\nabla}$ coincide), then they are $h$-projectively equivalent.  Indeed, the  equation (\ref{eq:eq1}) for the first and for the second metric coincide if $\nabla=\bar \nabla$.
\end{rem}

\begin{defn}
\label{def:hprotrafo}
Let $(M,g,J)$ be a K\"ahler manifold. A diffeomorphism $f:M\rightarrow M$ is called an \emph{$h$-projective transformation}, if $f$ is  \emph{ holomorphic} (that is, if  $f_*(J)= J$), 
 and if $f^{*}g$ is $h$-projectively equivalent to $g$. A vector
 field $v$ is called $h$-projective, if its local flow $\Phi^{v}_{t}$ consists of  (local)  $h$-projective transformations. Similarly, a diffeomorphism $f:M\rightarrow M$ is called an \emph{affine transformation}, if it preserves the Levi-Civita connection of $g$. A vector
 field $v$ is \emph{affine}, if its local flow  consists of  (local)  affine transformations. An $h$-projective transformation (resp. $h$-projective vector
field) is called \emph{essential}, if it is not  an affine transformation (resp. affine vector
field).
\end{defn}
Clearly, the set of all $h$-projective transformations of $(M,g,J)$ is a group. As it was shown in \cite{Ishihara1957} and \cite{Yoshimatsu}, it is a finite-dimensional Lie group. We denote it by  $\mbox{HProj}(g,J)$. By Remark \ref{rem:affishpro}, holomorphic affine transformations and holomorphic isometries are $h$-projective transformations, $\mbox{Iso}(g, J) \subseteq \mbox{Aff}(g,J)\subseteq\mbox{HProj}(g, J)$. Obviously, the same is  true for the connected components of these groups containing the identity transformation: $\mbox{Iso}_0(g, J) \subseteq \mbox{Aff}_0(g,J)\subseteq\mbox{HProj}_0(g, J)$.

\begin{exmp}[Generalisation of the Beltrami construction from \cite{Beltrami,M2005}]
\label{ex:cpn}
Consider a non-degenerate complex linear transformation $A\in Gl_{n+1}(\mathbb{C})$ and the induced bi-holomorphic diffeomorphism  $f_A:\mathbb{C}P(n) \to \mathbb{C}P(n)$.   Since the mapping $f_A$ sends   projective lines to projective lines, it sends   $h$-planar curves (of the Fubiny-Study metric $g_{FS}$) to $h$-planar curves, see Example \ref{ex:ex3}.  Then, the  
  pullback $g_A:= f_A^*g_{FS}$ is $h$-projectively equivalent to  $g_{FS}$ and  $f_A$ is an $h$-projective transformation. Note that the metric $g_A$ coincides with $g_{FS}$ (i.e., $f_A$ is  an isometry),
   if and only if $A$ is proportional to a unitary matrix. 
\end{exmp}

We see that for 
 $(\mathbb{C}P(n),g_{FS},J)$  we have $\mbox{Iso}_0 {\not =} \mbox{HProj}_0$.
 Our main result is

\begin{thm}[Yano-Obata conjecture]
\label{thm:obata}
Let $(M,g,J)$ be a closed, connected Riemannian K\"ahler manifold of real dimension $2n\geq 4$. Then, 
$\mbox{Iso}_0(g, J) = \mbox{HProj}_0(g, J)$  unless $(M,g,J)$ can be covered by $(\mathbb{C}P(n),c\cdot g_{FS},J)$ for some positive  constant $c$.
\end{thm}

\begin{rem} The above Theorem is not true locally; one can construct counterexamples. We conject  that Theorem  \ref{thm:obata} is also true if we replace closedness by completeness; but dealing with this 
 case will require a lot of  work. In particular, one will need to generalize the results of \cite{FKMR} to the complete metrics. 
\end{rem}

\subsection{History and motivation.}
\label{subsec:hismot}
$h$-projective equivalence was  introduced  by  Otsuki and Tashiro in  \cite{Otsuki1954,Tashiro1956}.
They have shown that the classical projective equivalence is not interesting in the  K\"ahler situation since only  simple examples are possible, and suggested $h$-projective equivalence as an  interesting object of study instead.  This suggestion  was very fruitful.   During 60th-70th,   the theory of $h$-projectively equivalent metrics and $h$-projective transformations  was  one of the main research topics  in Japanese and Soviet (mostly Odessa and Kazan) differential geometry schools, see for example the survey \cite{Mikes} with more than one hundred fifty references. Two classical  books  \cite{Sinjukov,Yanobook} contain chapters on $h$-projectivity.

New interests to $h$-projective equivalence is due to its connection with the so called {\it hamiltonian $2$-forms} defined  and investigated in Apostolov et al  \cite{ApostolovI,ApostolovII,ApostolovIII,ApostolovIV}. Actually, a  hamiltonian $2$-form   is essentially the same as a $h$-projectively equivalent 
metric $\bar g$: it is easy to  see that  the defining equation \cite[equation $(12)$]{ApostolovI} of a hamiltonian $2$-form is algebraically equivalent to the equation \eqref{f}, which is a reformulation of the condition ``$\bar g$ is $h$-projectively equivalent to $g$'' to the language of PDE, see Remark \ref{apostol}. The motivation of Apostolov et al  to study hamiltonian $2$-forms is different from that of Otsuki and Tashiro and is  explained in \cite{ApostolovI,ApostolovII}. Roughly speaking, they observed that many interesting problems on  Kähler manifolds lead to  hamiltonian $2$-forms and suggested to study them. The motivation is justified in \cite{ApostolovIII,ApostolovIV}, where they indeed constructed  new interesting and useful examples of Kähler manifolds. There is also a direct connection between $h$-projectively equivalent metrics and  conformal Killing  (or twistor) $2$-forms  studied in \cite{Moroianu,Semmelmann0,Semmelmann1}, see Appendix A of \cite{ApostolovI} for details. 

In private communications with the authors of \cite{ApostolovI,ApostolovII,ApostolovIII,ApostolovIV} we got informed that they did not know that the object they considered  was studied before under another name. Indeed, they re-derived certain facts that were well known in the theory of $h$-projectively equivalent metrics. 
On the other hand, the papers    \cite{ApostolovI,ApostolovII,ApostolovIII,ApostolovIV} contain several solutions of the problems studied in the framework of $h$-projectively equivalent metrics, for example the local \cite{ApostolovI} and global \cite{ApostolovII} description of $h$-projectively equivalent metrics --- previously, only special cases of such   descriptions (see, for example, \cite{Kiyo1997}) were known. 

Additional  interest to $h$-projectivity is  due to a connection between $h$-projectively equivalent metrics and integrable geodesic flows: it appears that the existence of $\bar g$ that is $h$-projectively equivalent to $g$ allows us to construct quadratic and
linear integrals for the geodesic flow of $g$. The existence of quadratic integrals was proved by Topalov \cite{Top2003}.  Under certain nondegeneracy assumptions,   the quadratic integrals of Topalov are as considered by Kiyohara in \cite{Kiyo1997}; the existence of such integrals immediately implies  the existence of Killing  vector fields. In the general situation, the existence of the Killing vector fields follows from \cite{ApostolovI} and, according to a private conversation, was known to Topalov. Altogether, in the most nondegenerate case studied by Kiyohara, we obtain $n$ quadratic and $n$ linear integrals on a $2n$-dimensional manifold; the integrals are in involution and are functionally independent so 
 the geodesic flow of the metric is Liouville-integrable.   In the present paper, we will actively use the quadratic integrals. We will also use one Killing vector field whose existence is well-known. 

Note that 
the attribution of the Yano-Obato conjecture  to Yano and Obata  is in folklore - we did not find a paper of them where they state  this conjecture.  It is clear though that both Obata and Yano (and many other geometers)  tried to prove this statement and did this     under certain additional assumptions, see below.  The conjectures of similar type were standard in 60th-70th, in the time when Yano and Obata were active (and it was also, unfortunately, standard in that time not to publish conjectures or open questions). For example, another famous conjecture of that time states that an essential group of conformal transformations of a Riemannian manifold is possible  if and only if the manifold is conformally equivalent to the standard sphere or to the Euclidean space; this conjecture is attributed to Lichnerowicz and Obata  though  it seems that neither Lichnerowicz nor Obata published it as a  conjecture or a question; it was solved in Alekseevskii \cite{Alekseevskii1972}, Ferrand \cite{Ferrand} and Schoen \cite{Schoen}.  One more example is the so-called  projective Lichnerowicz-Obata conjecture stating that  a  complete Riemannian manifold, such that the connected  component of  the neutral element of 
the projective group  contains not only   isometries,   has constant positive sectional curvature. This conjecture was proved in \cite{M2004bis,M2004,CMH,Matveev2007}. Though this conjecture is also attributed in folklore to Lichnerowicz and Obata, neither Lichnerowicz nor Obata  published  this conjecture (however, this particular conjecture  was published as   ``a classical conjecture'' in  \cite{hasegawa,nagano,Yamauchi1}). 	
In view of these two examples, it would be natural to call the Yano-Obata conjecture the Lichnerowicz-Obata conjecture for $h$-projective transformations.

 Special cases of Theorem \ref{thm:obata} were known before. For example, under the additional  assumption that the scalar curvature of $g$ is constant, the conjecture was proven in \cite{HiramatuK,Yano1981}. The  case when  the Ricci tensor of $g$  vanishes or  is  covariantly constant  was proven earlier in 
  \cite{Ishihara1957,Ishihara1960,Ishihara1961}.  Obata \cite{obata} and Tanno \cite{Tanno1978} proved this conjecture under the assumption that the $h$-projective vector field lies in the so-called $k$-nullity space of the curvature tensor. Many  local results related to  essential $h$-projective transformations are listed in the survey \cite{Mikes}. For example, in \cite{DomMik1978,Sakaguchi} it was shown that  locally symmetric spaces of non-constant holomorphic sectional curvature do not admit $h$-projective transformations, even locally.
  
 A very important special case of Theorem \ref{thm:obata}   was obtained in the recent paper  \cite{FKMR}.  Their, the Yano-Obata conjecture was proved under the additional  assumption that the degree of mobility (see Definition \ref{def:deg}) is $\ge 3$.  We will essentially use the results of \cite{FKMR} in our paper. Actually, we consider that both papers, \cite{FKMR} and the present one,  are equally important for the proof of the Yano-Obata conjecture.  The methods of \cite{FKMR} 
     came from the theory of overdetermined PDE-systems of finite type and are 
      very different from the methods of the present paper.

\section{The main equation of $h$-projective geometry and the scheme of the proof of Theorem \ref{thm:obata}}
\label{sec:main}
\subsection{Main equation of $h$-projective geometry.} 
Let $g$ and $\bar{g}$ be two Riemannian  (or pseudo-Riemannian) 
 metrics on $M^{2n\ge 4}$ that are K\"ahler with respect to the same complex structure $J$.  
  We  consider the induced isomorphisms $g:TM\rightarrow T^{*}M$ and $\bar{g}^{-1}:T^{*}M\rightarrow TM$. Let us  introduce the $(1,1)$-tensor $A(g,\bar{g})$  by the formula  
 \begin{equation}
  \label{eq:a}
  A(g,\bar{g})= \left(\frac{\det{\bar{g}}}{\det g}\right)^\frac{1}{2(n+1)}
   \bar{g}^{-1} \circ g: TM\to TM
\end{equation}
(in coordinates, the matrix of $\bar{g}^{-1} \circ g$ 
is the product of the inverse matrix of $\bar g$  and the matrix of $g$).

Obviously, $A(g,\bar{g})$ is non-degenerate, complex (in the sense that $A\circ J=J\circ A$) and self-adjoint with respect to both metrics. Let  $\nabla$ be the Levi-Civita connection of $g$. 
\begin{thm}[\cite{DomMik1978}] The metric 
  $\bar{g}$ is $h$-projectively equivalent to $g$, if and only if there exists a vector
  field $\Lambda$ such that $A=A(g,\bar{g})$ given by  \eqref{eq:a}  satisfies 
  \begin{equation}
\label{f}    
    (\nabla_{X}A)Y=g(Y,X)\Lambda + g(Y,\Lambda)X + g(Y,JX)\bar{\Lambda}+g(Y,\bar{\Lambda})JX,
  \end{equation}
for all $x \in M$ and all  $X,Y\in T_xM$, where $\bar{\Lambda}=J(\Lambda)$.
\end{thm}
\begin{rem}  
\label{rem:lambda2}
One may consider the  equation (\ref{f}) as a linear PDE-system on the unknown $(A,\Lambda)$; the coefficients in this system depend on the metric $g$. Indeed, 
if the  equation  is fulfilled for $X,Y$ being basis vectors, it is fulfilled for all vectors, see also the formula \eqref{eq:hpr} below. 

One can also consider equation \eqref{f} as a linear PDE-system on the $(1,1)$-tensor $A$ only, since the components of $\Lambda$ can be obtained from the components of $\nabla A$ by linear algebraic manipulations.
  Indeed, fix $X$ and calculate the trace of the $(1,1)$-tensors on the left and right-hand side of (\ref{f}). The trace of the right-hand side  equals $4g(\Lambda,X)$. Clearly, 
  the trace of $\nabla_{X}A$ is  $\trace(\nabla_{X}A)=X(\trace\,A)$. 
  Then,   $\Lambda=\grad\lambda$, where the function $\lambda$ is equal  to $\frac{1}{4}\trace\,A$.
In what follows, we prefer the last point of view and speak about  a self-adjoint, complex solution $A$ of equation (\ref{f}), instead of explicitly mentioning  the pair $(A,\Lambda)$. 
  \end{rem}
\begin{rem}
\label{rem:lambda}  
Let $g$ and $\bar{g}$ be two $h$-projectively equivalent K\"ahler metrics and $A(g,\bar{g})$ the corresponding solution of (\ref{f}).
It is easy to see that  $g$ and $\bar{g}$ are affinely equivalent, if and only if the corresponding vector
field $\Lambda$  vanishes  identically on $M$.
  \end{rem}
\begin{rem} The original and more standard form of the equation \eqref{f} uses index (tensor) notation and 
reads   \begin{equation}
    \label{eq:hpr}
    a_{i j,k}=\lambda_i g_{j k} + \lambda_j g_{i k} -
    \bar{\lambda}_{i}J_{jk}-\bar{\lambda}_{j}J_{ik}.
  \end{equation}
Here $a_{ij}$, $\lambda_i$ and $\bar \lambda_i$
 are related to $A, \Lambda$ and $\bar \Lambda$ by the formulas 
 $a_{ij}=g_{ip}A^{p}_{j}$, $\lambda_{i}=g_{ip}\Lambda^{p}$ and $\bar \lambda_{i}=-g_{ip}\bar \Lambda^{p}$. 
\end{rem}
\begin{rem} 
\label{rem:PDE} 
  Note that formula (\ref{eq:a}) is invertible if $A$ is non-degenerate: 
 the metric $\bar g$  can be reconstructed from $g$ and $A$ by 
 \begin{equation} \bar g=(\det\,A)^{-\frac{1}{2}}g\circ A^{-1}\label{inverse}  \end{equation} 
(we understand $g$ as the mapping $g:TM\to T^*M$; in coordinates, 
 the matrix of  $g\circ A^{-1}$ is the product of the matrices of $g$ and $A^{-1}$).

Evidently, if $A$ is $g$-self-adjoint and complex, 
$\bar g$  given by \eqref{inverse} is symmetric and invariant with respect to the complex structure.
It can be checked  by direct calculations that if $g$ is Kähler and if $A$ is a non-degenerate, $g$-self-adjoint and complex $(1,1)$-tensor satisfying (\ref{f}), then $ \bar g$ is also Kähler with respect to the same complex structure and is $h$-projectively equivalent to $g$.  
  
   Thus, the set of K\"ahler metrics, $h$-projectively equivalent to $g$, is essentially the same as the set of self-adjoint, complex (in the sense $J\circ A= A\circ J$)   solutions of (\ref{f}) (the only difference is the case when $A$ is degenerate, but since adding $\mbox{const}\cdot \Id$ to $A$ does not change the property of $A$ to be a solution, this difference is not important).
\end{rem}

\begin{rem} \label{apostol} 
As we have already mentioned in Section \ref{subsec:hismot}, equation \eqref{f} is equivalent to the defining equation for a hamiltonian $2$-form (see \cite[equation $(12)$]{ApostolovI}). Indeed, for a complex and self-adjoint solution $A$ of \eqref{f},  the  $2$-form $\Phi(X,Y):=g(JAX,Y)$ is hamiltonian in the sense of \cite{ApostolovI}.
  \end{rem}

By  Remark \ref{rem:lambda2}, the  equation (\ref{f}) is a system of linear PDEs  on the $(1,1)$-tensor $A$. 
\begin{defn}
\label{def:deg} We  denote by $\mbox{Sol}(g)$ the linear space of complex, self-adjoint solutions of equation (\ref{f}).
The \emph{degree of mobility $D(g)$} of a K\"ahler metric $g$ is  the dimension of the space $\mbox{Sol}(g)$.
\end{defn}
\begin{rem}
We note that $1\leq D(g)<\infty$. Indeed, since $\Id$ is always a solution of equation (\ref{f}), we have $D(g)\ge 1$. We will not use the fact that $D(g) <\infty$; a proof of this statement can be found in \cite{FKMR} or in \cite{DomMik1978}. 
\end{rem}
Let us now show that the degree of mobility is the same for two $h$-projectively equivalent metrics: we construct an explicit isomorphism. 
\begin{lem}
\label{lem:degree}
Let $g$ and $\bar{g}$ be two $h$-projectively equivalent K\"ahler metrics on $(M,J)$. Then the solution spaces $\mbox{Sol}(g)$ and $\mbox{Sol}(\bar{g})$ are isomorphic. The isomorphism is given by
\begin{align}
A_{1}\in \mbox{Sol}(g)\longmapsto A_{1}\circ A(g,\bar{g})^{-1}\in\mbox{Sol}(\bar{g}),\nonumber
\end{align}
 where $A(g,\bar{g})$ is constructed by \eqref{eq:a}. In particular, $D(g)$ is equal to $D(\bar{g})$. 
\end{lem}
\begin{proof}
Let $A=A(g,\bar{g})$ be the solution of (\ref{f}) constucted by formula \eqref{eq:a}. If $A_{1}\in \mbox{Sol}(g)$ is non-degenerate, then $g_{1}=(\det\,A_{1})^{-\frac{1}{2}}g\circ A_{1}^{-1}$  is $h$-projectively equivalent to $g$  by Remark \eqref{rem:PDE} and,  hence, $g_{1}$ is $h$-projectively equivalent to $\bar{g}$. It follows that $A_{2}=A(\bar{g},g_{1})\in\mbox{Sol}(\bar{g})$. On the other hand, using formula \eqref{eq:a} we can easily verify that $A_{2}=A_{1}\circ A^{-1}$. If $A_{1}$ is degenerate, we can choose a real number $t$ such that $A_{1}+t \Id$ is non-degenerate. As we have already shown, $(A_{1}+t \Id)\circ A^{-1}=A_{1}\circ A^{-1}+t A^{-1}$ is contained in $\mbox{Sol}(\bar{g})$. Since $A^{-1}\in\mbox{Sol}(\bar{g})$,  the same is true for $A_{1}\circ A^{-1}$. We obtain that the mapping $A_{1}\longmapsto A_{1}\circ A(g,\bar{g})^{-1}$ is a linear isomorphism between the spaces $\mbox{Sol}(g)$ and $\mbox{Sol}(\bar{g})$. 
\end{proof}

\begin{lem}[Folklore]
\label{lem:htrafo}
Let $(M,g,J)$ be a Kähler manifold and let $v$ be an $h$-projective vector
field. 
Then  the $(1,1)$-tensor 
\begin{equation} \label{hprojective vector field}
A_{v}:=g^{-1}\circ\mathcal{L}_{v}g-\frac{\trace(g^{-1}\circ\mathcal{L}_{v}g)}{2(n+1)}\Id
\end{equation}
(where $\mathcal{L}_{v}$ is the Lie derivative with respect to $v$) is contained in $\mbox{Sol}(g)$.
\end{lem}
\begin{proof}
Since $v$ is $h$-projective,  $\bar{g}_{t}=(\Phi^{v}_{t})^{*}g$ is $h$-projectively equivalent to $g$ for every $t$. It follows that for every $t$ the tensor $A_{t}=A(g,\bar{g}_{t})$ is a solution of equation (\ref{f}). Since $(\ref{f})$ is linear, $A_v:= \left(\frac{d}{dt}A_{t}\right)_{|_{t=0}}$ is also a solution of \eqref{f} and it is clearly self-adjoint. Since the flow of $v$ preserves the complex structure, $A_v$ is complex. 
  Using equation \eqref{eq:a}, we obtain  that $A_{v}$ is equal to  
\begin{align}
&\frac{d}{dt}\left[\left(\frac{\det\,\bar{g}_{t}}{\det\,g}\right)^{\frac{1}{2(n+1)}}\bar{g}_{t}^{-1}\circ g\right]{\Big|_{t=0}}=\frac{1}{2(n+1)}\left(\frac{d}{dt}\frac{\det\,\bar{g}_{t}}{\det\,g}\Big|_{t=0}\right)\Id
+\left(\frac{d}{dt}\bar{g}_{t}^{-1}\circ g\right)\Big|_{t=0}\nonumber\vspace{4mm}\\
&=\frac{1}{2(n+1)}\left(\frac{d}{dt}\frac{\det\,\bar{g}_{t}}{\det\,g}\right)\Big|_{t=0}\Id-\left(\bar{g}_{t}^{-1}\circ\left(\frac{d}{dt}\bar{g}_{t}\right)\circ{\bar g}_{t}^{-1}\circ g\right)\Big|_{t=0}\nonumber\vspace{4mm}\\
&=\frac{1}{2(n+1)}\left(\frac{d}{dt}\frac{\det\,\bar{g}_{t}}{\det\,g}\right)\Big|_{t=0}\Id+g^{-1}\circ\mathcal{L}_{v}g=-\frac{\trace(g^{-1}\circ\mathcal{L}_{v}g)}{2(n+1)}+g^{-1}\circ\mathcal{L}_{v}g.
\nonumber
\end{align}
Thus,  $A_{v}\in Sol(g)$ as we claimed.
\end{proof}

\subsection{Scheme of the proof of Theorem \ref{thm:obata}.}
\label{sec:scheme}

 If the  degree of mobility  $D(g)\geq 3$, Theorem \ref{thm:obata} is an immediate consequence of \cite[Theorem 1]{FKMR}. Indeed, by \cite[Theorem 1]{FKMR}, if $D(g)\geq 3$ and the manifold cannot be covered by $(\mathbb{C}P(n),c\cdot g_{FS},J)$ for some $c>0$, every metric $\bar g$  that is  $h$-projectively equivalent to $g$ is actually affinely  equivalent to $g$. By  \cite{lichnerowicz}, the connected component of the neutral element of the group of affine transformations on a closed manifold consists of isometries. Finally, we obtain $\mbox{HProj}_0= \mbox{Iso}_0$. 
 
  If the  degree of mobility is  equal to one, every metric $\bar{g}$ that is  $h$-projectively equivalent to $g$ is proportional to $g$. Then, the group $\mbox{HProj}_0(g, J)$ acts by homotheties. Since the manifold is closed, it acts by isometries. Again, we obtain $\mbox{HProj}_0= \mbox{Iso}_0$. 
 
 Thus,  in the proof of Theorem \ref{thm:obata}, we may (and will)   assume that the degree of mobility of the metric $g$ is equal to two.  
 
 The proof will be organized as follows. 
In  Sections \ref{sec:quadintglobord} and \ref{sec:basics}, we collect and proof basic facts  that will be used in the  proof of   Theorem \ref{thm:obata}.   Certain results of Sections \ref{sec:quadintglobord},  \ref{sec:basics}  were known before; we will give precise references. 
Proofs in  Sections \ref{sec:quadintglobord} and \ref{sec:basics}  are based on different groups of methods and different ideas.  In Section \ref{sec:quadintglobord},  we use the familiy of 
  integrals  for the geodesic flow of the metric $g$ found  by Topalov \cite{Top2003}. With the help of  these integrals, we prove that the eigenvalues of $A$ behave quite regular, in particular we show that they are globally ordered and that the multiplicity of every nonconstant eigenvalue is equal to two. The assumptions of this section are  global (we assume that every two points can be connected  by a geodesic).

  In Section \ref{sec:basics}, we work locally  with equation
\eqref{f}. We  show that $\Lambda$ and $\bar \Lambda$  from this equation are commuting holomorphic vector fields that are nonzero at almost every point. We also deduce from \eqref{f} certain  equations on the eigenvectors and eigenvalues of $A$:  in particular we show that  the gradient 
 of every  eigenvalue is an eigenvector corresponding to this eigenvalue.

Beginning with Section \ref{sec:D=2}, we require the assumption that the degree of mobility is equal to two. Moreover, we assume the existence of an $h$-projective vector field which is not already an affine vector field. The main goal of  Section \ref{sec:D=2} is to show that for every 
 solution $A$ of equation (\ref{f}) with the corresponding vector field $\Lambda$, in a neighborhood of almost every point 
  there exists  a function $\mu$ and a constant $B<0$ (that can depend on the neighborhood)
   such that for all points $x$ in this neighborhood and all $X,Y\in T_xM $ we have 
\begin{align}
\begin{array}{c}
(\nabla_{X}A)Y=g(Y,X)\Lambda + g(Y,\Lambda)X + g(Y,JX)\bar{\Lambda}+g(Y,\bar{\Lambda})JX\vspace{2mm}\\
\nabla_{X}\Lambda=\mu X+B A(X)\vspace{2mm}\\
\nabla_{X}\mu=2Bg(X,\Lambda)
\end{array}\label{eq:systemscheme}
\end{align}
(one should view \eqref{eq:systemscheme} as a PDE-system on $(A, \Lambda, \mu)$).

This is the longest and the most complicated part of the proof. First, in Section \ref{sec:onenonconstant}, we combine Lemma  \ref{lem:htrafo}  with the assumption that the  degree of mobility is two to obtain the formulas (\ref{eq:liederivA},\ref{eq:lie}) that describe the evolution of $A$ along the flow of the $h$-projective vector field. With the help of the results of Section \ref{sec:basics}, we deduce  (in the proof of Lemma \ref{lem:spectrumD=2}) an  ODE for the eigenvalues of $A$ along the trajectories of the $h$-projective vector field. This ODE can be solved; combining the solutions \weg{with the assumption that the manifold is closed and} with the global ordering of the eigenvalues from Section \ref{sec:quadintglobord}, we obtain that $A$ has at most three eigenvalues at every point; moreover, precisely one  eigenvalue of $A$ considered as a function on the manifold is not      constant  (unless the $h$-projective vector field is an affine vector field).  As a consequence,  in view of the  results of Section \ref{sec:basics}, the vectors  $\Lambda$ and $\bar{\Lambda}$ are eigenvectors of $A$.

The equation \eqref{eq:lie} depends on two parameters. We prove that under the assumption that the manifold is closed, the parameters are subject of a certain algebraic equation so that in fact the  equation \eqref{eq:lie} depends on one parameter only. In order to do it, we work with the distribution $\spann\{\Lambda, \bar\Lambda\}$ and show that its integral manifolds are totally geodesic. The equations  (\ref{hprojective vector field},\ref{eq:lie}) contain enough information to calculate the  restriction of the metric to this distribution; the metric depends on the same  parameters  as  the equation \eqref{eq:lie}. We calculate the sectional curvature of this metric and see that it is unbounded (which can not happen on a closed manifold), unless the parameters satisfy a certain  algebraic equation.

 In  Section \ref{sec:systemloc},  we show that  the algebraic  
  conditions mentioned above imply the local existence of $B$ and $\mu$ 
 such that  \eqref{eq:systemscheme} is fulfilled. This proves that the system \eqref{eq:systemscheme} is satisfied in a neighborhood of almost every point of $M$, for certain $B, \mu$ that can a priori depend on the neighborhood.

We complete  the proof of Theorem \ref{thm:obata}  in  Section \ref{sec:obata}. First we recall  certain  results of \cite{FKMR} to show that the constant $B$  is the same in all neighborhoods implying that the system  \eqref{eq:systemscheme} is fulfilled on the whole manifold.

Once we have shown that   the system \eqref{eq:systemscheme} holds globally,  Theorem \ref{thm:obata} is an immediate consequence of   \cite[Theorem 10.1]{Tanno1978}.

\subsection{Relation with projective equivalence.}
\label{relation}

Two metrics $g$ and $\bar g$ on the same manifold are
\emph{projectively equivalent}, if every geodesic of $g$, after an
appropriate reparametrization, is a geodesic of $\bar g$. As we
already mentioned above,  the notion
``$h$-projective equivalence'' appeared as an attempt to adapt
the notion ``projective equivalence'' to Kähler metrics. It is
therefore not a surprise that certain methods from the theory of
projectively equivalent metrics could be adapted to the $h$-projective
situation. For example, the above mentioned papers
\cite{HiramatuK,Yano1981,Akbar} are actually  $h$-projective analogs
of the papers~\cite{Yamauchi1,Hiramatu} (dealing with projective
transformations), see also~\cite{Hasegawa,Takeda}.  Moreover,
\cite{Yoshimatsu,Tashiro1956} are $h$-projective analogs  of
\cite{Ishihara1961,Tanaka}, and many results listed in 
the survey \cite{Mikes} are
$h$-projective analogs  of those listed in \cite{Mikes1996}.

 The Yano-Obata conjecture is also an $h$-projective analog of the
so-called projective Lichnerowicz-Obata conjecture mentioned above and recently proved in
\cite{Matveev2007,CMH}, see also~\cite{M2004,M2004bis}. 
The general scheme of our 
 proof of the Yano-Obata conjecture 
 is similar to the scheme of  the proof of the projective Lichnerowicz-Obata conjecture in 
\cite{Matveev2007}. More precisely,  as in the projective case, the cases degree of mobility equal to two and degree of mobility $\ge 3$ were done using completely different groups of methods.  As we mentioned above,   
the proof of the Yano-Obata conjecture 
for the metrics with degree of mobility $\ge 3$  was  done in \cite{FKMR}. This proof is based on other ideas than the corresponding part of the proofs of the projective Lichnerowicz-Obata conjecture in \cite{Matveev2007,M2006}.

 Concerning the proof under the assumption that  the degree of 
mobility is two, the first part of the proof  (Sections \ref{sec:quadintglobord}, \ref{sec:onenonconstant}) 
 is based on  the same ideas as in the projective case.   
More precisely, the way to use integrals for the geodesic flow to show the regular behavior of the  eigenvalues of $A$ and their global ordering  is very close to that of  \cite{BM2003,Inventiones2003,zametki,dedicata}. The way to obtain equation  \eqref{eq:lie} that describes the evolution of $A$ along the orbits of the $h$-projective vector field is close to that in \cite{splitting} and is  motivated by \cite{M2004,M2004bis,CMH,Matveev2007}.

\weg{ Concerning further perspective  directions of investigation, one can try to obtain topological  obstructions that prevent a 
closed manifold to posses  two  $h$-projectively equivalent
metrics that are not affinely equivalent.  In the projective case,  the proof of 
certain topological restrictions is based on the integrals for the geodesic flow, see  \cite{M2002,M2003,Inventiones2003,M2005}; we expect that these methods can be generalized for the $h$-projective case. }

\section{Quadratic integrals and the global ordering of the eigenvalues of solutions of equation (\ref{f})}
\label{sec:quadintglobord}
\subsection{Quadratic integrals for the geodesic flow of $g$.}
\label{sec:int}

Let $A$ be  a self-adjoint, complex solution of equation (\ref{f}). By \cite{Top2003} (see also 
 the end of Appendix A of \cite{ApostolovI}), for every $t\in \mathbb{R}$,  the function  
\begin{align}
F_{t}:TM\to \mathbb{R}\, , \ \  F_{t}(\zeta):=\sqrt{\det\,(A-t\Id)}\,g((A-t\Id)^{-1}\zeta,\zeta)\label{eq:int2}
\end{align}
is an  integral for the geodesic flow of $g$.  

\begin{rem} \label{eas} 
It is  easy to prove  (see formula \eqref{eq:int2expl} below)  that 
 the integrals are defined for all $t\in\mathbb{R}$ (i.e., even  if $A-t\Id$ is degenerate).
 Actually, the family $F_t$ is a polynomial of degree $n-1$ in $t$ whose coefficients are certain functions on $TM$; these functions are automatically integrals. 
\end{rem} 

\begin{rem} 
The integrals are visually close to the integrals for the geodesic flows of projectively equivalent metrics constructed in \cite{MT}.
\end{rem} 

Later it will be useful to consider derivatives of the integrals defined above:
\begin{lem}
\label{lem:dint}
Let $\{F_{t}\}$ be the family of integrals given in equation (\ref{eq:int2}). Then, for each integer $m\geq 0$ and for each number $t_0\in\mathbb{R}$,
\begin{equation} \label{der} 
\left(\tfrac{d^{m}}{dt^{m}}F_{t}\right)_{|t=t_0}
\end{equation}
is also an integral for the geodesic flow of $g$.
\end{lem}
\begin{proof}
 As we already mentioned above in Remark  \ref{eas}, 
\begin{align}
F_{t}(\zeta)=s_{n-1}(\zeta) t^{n-1}+...+s_{1}(\zeta) t+s_{0}(\zeta)\nonumber
\end{align}
for certain integrals  $s_{0},...,s_{n-1}:TM\to \mathbb{R}$.  Then, the $t$-derivatives \eqref{der} are 
 also  polynomials in $t$ whose coefficients are integrals, i.e., the $t$-derivatives \eqref{der} are also integrals for every fixed $t_0$. 
\end{proof} 

\subsection{Global ordering of the eigenvalues of solutions of equation (\ref{f})}
\label{sec:globord}
During the whole subsection let $A$ be an element of $\mbox{Sol}(g)$; that is, $A$ is a complex, self-adjoint $(1,1)$-tensor such that it is a solution of equation (\ref{f}). Since it is self-adjoint with respect to (a positively-definite metric) $g$, 
the eigenvalues of $A_{|x}:=A_{|T_{x}M}$ are  real.

\begin{defn}
\label{defn:stable}
We denote by $m(y)$ the number of different eigenvalues of $A$ at the point $y$. Since $A\circ J=J\circ A$, each eigenvalue has even  multiplicity $\geq 2$. Hence, $m(y) \leq n$ for all $y\in M$.
We say that $x\in M$ is a \textit{typical point } for $A$ if $m(x)=\max_{y\in M}\{m(y)\}$. 
The set of all typical points of $A$ will be denoted by $ M^{0}\subseteq M$. 
\end{defn}

 Let us denote by 
 $\mu_{1}(x)\leq...\leq\mu_{n}(x)$ the eigenvalues of $A$ counted with  half of their multiplicities. 
 The functions $\mu_{1},...,\mu_{n}$ are real since $A$ is self-adjoint and they are at least continuous.  It follows that $M^{0}\subseteq M$ is an open subset. 
 The next  theorem  shows that  $M^{0}$ is dense in $M$.
\begin{thm}
\label{thm:Mat}
Let $(M,g,J)$ be a  Riemannian  K\"ahler manifold of real dimension $2n$.  Suppose every two points of $M$ can be connected by a geodesic.  Then, for every  $A\in \mbox{Sol}(g)$ and every   $i=1,...,n-1$,  the following statements hold:
\begin{enumerate}
\item $\mu_{i}(x)\leq\mu_{i+1}(y)$ for all $x,y\in M$.\vspace{2mm}
\item If $\mu_{i}(x)<\mu_{i+1}(x)$ at least at one point, then  the set of all points $y$ such that $\mu_{i}(y)<\mu_{i+1}(y)$ is everywhere dense  in $ M$.
\end{enumerate}
\end{thm}

\begin{rem}
If the  Kähler manifold is compact, the global description of  hamiltonian $2$-forms \cite[Theorem 5]{ApostolovII} implies the global ordering of the eigenvalues (the first part of Theorem \ref{thm:Mat}), and this is sufficient for our further goals. However, we give an alternative proof which works under less general assumptions, and is based on other ideas. 
\end{rem}

\begin{proof}
$(1)$: Let $x\in M$ be an arbitrary point. At $ T_xM$, 
 we choose 
an orthonormal frame $\{U_{i},JU_{i}\}_{i=1,...,n} $ of eigenvectors   (we assume  
 $AU_{i}=\mu_{i}U_{i}$  and $g(U_i, U_i)=1$ for all $i=1,...,n$). For $\zeta\in T_{x}M$,  we denote its components in the  frame $\{U_{i},JU_{i}\}_{i=1,...,n}$ by
   $\zeta_{j}:=g(\zeta,U_{j})$ and $\bar{\zeta}_{j}:=g(\zeta,JU_{j})$. 
   By direct calculations, we see that  
    $F_{t}(\zeta)$ given by  \eqref{eq:int2} reads 
\begin{equation} \begin{array}{rl} 
F_{t}(\zeta)=& \sum\limits_{i=1}^n \Big[(\zeta_{i}^{2}+\bar{\zeta}_{i}^{2})\prod\limits_{j=1; j\ne i}^n (\mu_{j}-\mu_{i})\Big]
\\=&(\mu_{2}-t)\cdot...\cdot(\mu_{n}-t) (\zeta_{1}^{2}+\bar{\zeta}_{1}^{2})+...+(\mu_{1}-t)\cdot...\cdot(\mu_{n-1}-t) (\zeta_{n}^{2}+\bar{\zeta}_{n}^{2}).\end{array}\label{eq:int2expl}
\end{equation} 
Obviously, $F_{t}(\zeta)$ is a polynomial in $t$ of degree $n-1$ whose leading coefficient is $(-1)^{n-1}g(\zeta, \zeta)$. 

For every point $x\in M$ and every $\zeta\in T_xM$ such that $\zeta\ne 0$, let us consider the roots $$t_1(x, \zeta), ... , t_{n-1}(x, \zeta):T_xM\to \mathbb{R} $$ 
of the polynomial counted with their multiplicities. From the arguments below it will be clear that they are real. We assume 
that at every $(x,\zeta)$  we have $t_1(x, \zeta)\le  ... \le  t_{n-1}(x, \zeta)$. Since for every fixed  $t$ 
 the polynomial $F_t$ is an integral, the roots $t_i$ are also integrals. 

Let us show that for every $i=1,...,n-1$ the inequality

\begin{equation}\label{eq:eigval}
 \mu_i(x) \le t_i(x, \zeta)\le  \mu_{i+1}(x)
\end{equation} 
holds.

We consider first the case when 
  all eigenvalues are different from each other,  i.e., $\mu_{1}(x)<...<\mu_{n}(x)$, and all components $\zeta_i$ are different from zero.   Substituting   $t=\mu_{i}$ and $t=\mu_{i+1}$ in equation \eqref{eq:int2expl}, we obtain   
\begin{align}
F_{\mu_{i}}(\zeta)=(\mu_{1}-\mu_{i})\cdot...\cdot(\mu_{i-1}-\mu_{i})\cdot (\mu_{i+1}-\mu_{i})\cdot...\cdot(\mu_{n}-\mu_{i})(\zeta_{i}^{2}+\bar{\zeta}_{i}^{2}),\nonumber
\end{align}
\begin{align}
F_{\mu_{i+1}}(\zeta)=(\mu_{1}-\mu_{i+1})\cdot...\cdot(\mu_{i}-\mu_{i+1})\cdot (\mu_{i+2}-\mu_{i+1})\cdot...\cdot(\mu_{n}-\mu_{i+1})(\zeta_{i+1}^{2}+\bar{\zeta}_{i+1}^{2}).\nonumber
\end{align}
We see that $F_{\mu_i}(\zeta)$ and $F_{\mu_{i+1}}(\zeta)$
 have different signs, see figure \ref{1}.  Then, every  open interval  $(\mu_{i},\mu_{i+1})$ contains a root of the polynomial $F_{t}(\zeta)$. Thus, all $n-1$ roots of the polynomial are real, and the inequality \eqref{eq:eigval} holds as we claimed. 
 
 \begin{figure}
  \includegraphics[width=.9\textwidth]{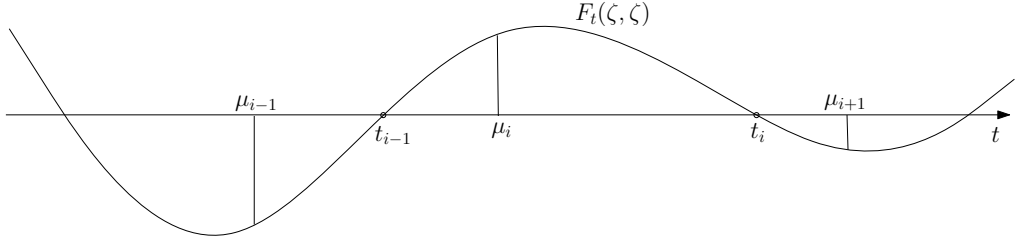}
  \caption{If $\mu_1<\mu_2<...<\mu_n$ and all  $\zeta_i\ne 0$, the values of $F_t(\zeta)$ have different signs at $t=\mu_i$ and $t=\mu_{i+1}$ implying the existence of a root $t_i$ such that $\mu_i< t_i < \mu_{i+1}$. }\label{1}
\end{figure}

  In the general case, since  $F_{t}(\zeta)$ depends continuously on the vector $\zeta$ and on the eigenvalues $\mu_{1}(x)\leq...\leq\mu_{n}(x)$ of $A_{|x}$, its zeros also depend continuously on $\zeta$ and $\mu_i$. It follows that for every $x$  and for all $\zeta\in T_{x}M$   we have   that all zeros are real and that 
\eqref{eq:eigval} holds.

Let us now show that for any two points $x,y$ we have 
$\mu_i(x)\le \mu_{i+1}(y)$. 

 We consider a geodesic  $\gamma:[0,1]\rightarrow M$
  such that $\gamma(0)=x$ and $\gamma(1)=y$. 
Since $F_{t}$ are integrals, we have $F_{t}(\dot\gamma(0))=F_{t}(\dot\gamma(1))$ implying 
\begin{equation} \label{tt}
 t_i(\gamma(0), \dot\gamma(0)) = t_{i}(\gamma(1), \dot\gamma(1)).\end{equation}

 Combining \eqref{eq:eigval} and  \eqref{tt}, we obtain 
  $$\mu_{i}(x)\stackrel{\eqref{eq:eigval}}{\leq} t_{i}(x,\dot\gamma(0))\stackrel{\eqref{tt}}{=}t_{i}(y,\dot\gamma(1))\stackrel{\eqref{eq:eigval}}{\leq} \mu_{i+1}(y)$$
   which proves the first part of Theorem \ref{thm:Mat}.
   
$(2)$: Assume $\mu_{i}(y)=\mu_{i+1}(y)$ for all points $y$ in some nonempty open subset $U\subseteq M$. We need to prove that for every $x\in M$ we have  $\mu_{i}(x)=\mu_{i+1}(x)$.

First let us show that $\mu:=\mu_{i}=\mu_{i+1}$ is a constant on $U$. Indeed, suppose that $\mu_{i}(y_{1})\leq\mu_{i}(y_{2})$ for some points $y_{1},y_{2}\in U$. From the first part of Theorem \ref{thm:Mat}  and from the assumption $\mu_i=\mu_{i+1}$ we  obtain 
$$\mu_{i}(y_{1})\leq\mu_i(y_2)\leq\mu_{i+1}(y_1) = \mu_i(y_1)$$ 
implying   $\mu_{i}(y_{1})=\mu_{i}(y_{2})$ for all $y_{1},y_{2}\in U$ as we claimed.

Now take an arbitrary point $x\in M$ and   consider  the set of  all initial velocities of geodesics  connecting $x$ with points of $U$ (we assume $\gamma(0)=x$ and $\gamma(1)\in U$), see figure \ref{2}.
\begin{figure}
  \centering
  \includegraphics[width=.6\textwidth]{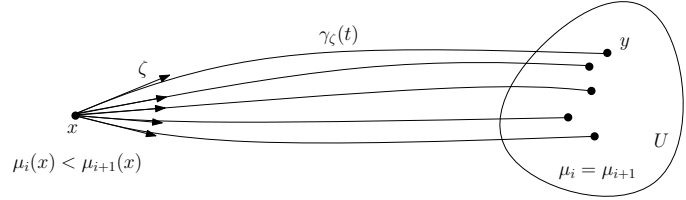}
  \caption{ 
   The initial velocity  vectors $\zeta$ at $x$ of the  geodesics connecting 
    the point $x$ with points from $U$ form a subset of nonzero measure and are contained in $U_{\mu}$.}
  \label{2}
\end{figure}
   For every such geodesic $\gamma$  we have  
$$\mu =\mu_i(\gamma(1)) \le t_i(\gamma(1), \dot\gamma(1)) \le \mu_{i+1}(\gamma(1))=\mu.$$ 
Thus, $t_i(\gamma(1), \dot\gamma(1))=\mu$.  
Since the value  $t_i(\gamma(t), \dot\gamma(t))$ is the same for all points of the geodesic, we obtain that   $t_i(\gamma(0), \dot\gamma(0))=\mu$. Then, 
the set 
\begin{align}
U_{\mu}:=\{\zeta\in T_{x}M:\  t_{i}(x,\zeta)=\mu\}\nonumber
\end{align} 
has nonzero measure.  Since $U_{\mu}$ is
contained  in the set 
$$
\{\zeta\in T_{x}M:\ F_\mu(\zeta) =0\}$$
which is a quadric in $T_{x}M$, the latter must coincide with the whole $T_{x}M$. In view of formula  \eqref{eq:int2expl}, this implies that at least two eigenvalues of $A$ at $x$ should be equal to $\mu$.
Suppose the multiplicity of the eigenvalue $\mu$ is equal to $2k$. This implies that $\mu_{r+1}(x)=...=\mu_{r+k}(x)=\mu$, $\mu_{r}(x)\ne \mu$ and $\mu_{r+ k +1}(x)\ne \mu$. If $i\in \{r+1,...,r+k-1\}$, we are done.  We assume that  $i \not \in \{r+1,...,r+k-1\}$ and find a contradiction.

In order to do it,  we consider the function  
$$\tilde F:\mathbb{R}\times TM\to \mathbb{R}\,, \ \ \tilde F_t(\zeta):= F_t(\zeta)/(t-\mu)^{k-1}.$$
 At the point $x$, each term of the sum \eqref{eq:int2expl} contains $(t-\mu)^{k-1}$ implying that $\tilde F_t(\zeta)$ is a polynomial  in $t$ (and is a quadratic function in $\zeta$). Since for every fixed $t_0$ the function $F_{t_0}$ is an integral, the function $\tilde F_{t_0}$ is also an integral.  
  Let us show  that for every geodesic $\gamma$ with $\gamma(0)=x$ and $\gamma(1)\in U$ we have 
   that $\left(\tilde F_t(\dot\gamma(0))\right)_{|t=\mu}=0$. Indeed, we already have shown that 
   $t_i(x, \dot\gamma(0))=\mu$. By similar arguments, in view of inequality \eqref{eq:eigval}, we  obtain 
 $t_{r+1}(x, \dot\gamma(0))=...=t_{r+k-1}(x, \dot\gamma(0))=\mu$. Then, $t=\mu$ is a root of multiplicity $k$ of $F_t(\dot\gamma(0))$ and, therefore, a root of multiplicity $k-(k-1)= 1$ of $\tilde F_t(\dot\gamma(0))= F_t(\dot\gamma(0))/(t-\mu)^{k-1}$.
 Finally,  $\tilde F_\mu(\dot\gamma(0))=0$.

   Now, in view of the formula  \eqref{eq:int2expl}, the set $\{ \zeta\in T_xM : \   \tilde F_\mu(\zeta)=0\}$ is a nontrivial (since $\mu_{r}\ne \mu \ne \mu_{r+k+1}$)   quadric in $T_xM$, 
   which contradicts the assumption that it contains a subset  $U_\mu$ of nonzero measure. Finally, we have $i,i+1\in \{r+1,...,r+k\}$ implying $\mu_i(x)= \mu_{i+1}(x)=\mu$. 
\end{proof}

From Theorem \ref{thm:Mat}, we immediately obtain the following two corollaries:
\begin{cor}
\label{cor:m0isdense}
Let $(M,g,J)$ be a complete, connected Riemannian K\"ahler manifold. Then, for every $A\in \mbox{Sol}(g)$, 
 the set $M^{0}$ of typical points of $A$ is open and dense in $M$.
\end{cor}

\begin{cor}[\cite{ApostolovI}]
\label{cor:mult}
Let $(M,g,J)$ be a complete, connected Riemannian K\"ahler manifold and assume $A\in \mbox{Sol}(g)$. Then, at almost every point 
 the multiplicity of a non-constant eigenvalue $\rho$ of $A$ is equal to two. 
\end{cor}

\section{Basic properties of solutions $A$ of equation (\ref{f})}
\label{sec:basics}
{ In this section, we  collect some basic technical properties of solutions of equation \eqref{f}. Most of the results are known in folklore; we will give precise references wherever possible.

\subsection{The vector fields $\Lambda$ and $\bar{\Lambda}$ are holomorphic}
\label{sec:Killing}
\begin{lem}[Folklore,  see equation (13) and the sentence below in \cite{DomMik1978}, Proposition 3 of \cite{ApostolovI} and Corollary 3  of \cite{FKMR}]
 \label{lem:Killing}
Let $(M,g,J)$ be a K\"ahler manifold of real dimension $2n\ge 4$ and let be $A\in\mbox{Sol}(g)$. Let $\Lambda$ be the corresponding vector field defined by equation (\ref{f}). Then $\bar{\Lambda}$ is a Killing vector field for the K\"ahler metric $g$,  i.e.,
\begin{align}
g(\nabla_{X}\bar{\Lambda},Y)+g(X,\nabla_{Y}\bar{\Lambda})=0\nonumber
\end{align}
for all $X,Y\in TM$.
\end{lem}
It is a well-known fact that if a Killing vector field $K$  vanishes on some open  nonempty subset $U$ of the connected manifold $M$, then $K$  vanishes on the whole  $M$. From this, we  conclude
\begin{cor}
\label{cor:htrafo2}
Let $(M,g,J)$ be a connected K\"ahler manifold of real dimension $2n\geq 4$ and let $v$ be an  $h$-projective vector field.
\begin{enumerate}
\item If $v$ restricted to  some open  nonempty subset $U\subseteq M$ is a Killing vector field, then $v$ is a Killing vector field on the whole  $M$. 
\item If $v$ is not identically zero, the set of points $M_{v\neq0}:=\{x\in M:v(x)\neq 0\}$ is open and dense in $M$.
\end{enumerate}
\end{cor}
\begin{proof}
$(1)$ Suppose  the restriction of $v$ to an  open subset $U$ is a Killing vector field. Then,  $\bar{g}_{t}=(\Phi^{v}_{t})^{*}g$ restricted to $U'\subset U$ is equal to $g_{|U'}$ for sufficiently small  $t$.  Hence, $A_{t|U'}=A(g,\bar{g}_{t})_{|U'}=\Id$. The corresponding vector field $\Lambda_t= \tfrac{1}{4} \grad\, \trace\,{A_t}$ vanishes (on $U'$)   implying  $\bar{\Lambda}_{t}$ vanishes  (on $U'$) as well.    Since $\bar{\Lambda}_{t}$ is a Killing vector field,   $\bar{\Lambda}_{t}$ vanishes  on the whole manifold implying   $\Lambda_{t}$ is equal to zero on the whole  $M$. Then, by \eqref{f},   the $(1,1)$-tensor  $A_{t}-\Id$ is  covariantly constant  on  the whole $M$. Since this tensor vanishes on   $U'$, it vanishes on the whole manifold. Finally, $A_{t}=\Id$ on $M$, implying that $v$ is a Killing vector field on  $M$. This proves part $(1)$ of Corollary \ref{cor:htrafo2}. 

$(2)$ Suppose  $v$ vanishes on some open subset $U\subseteq{M}$. To prove $(2)$, we have to show that $v=0$ everywhere on $M$. From part $(1)$ we can conclude that $v$ is a Killing vector field on $M$. Since $v$ vanishes on (open, nonempty) $U$, it  vanishes on the whole  $M$.
\end{proof}

The next lemma, which is  a wellknown and  standard result in Kähler geometry (we give a proof for self-containedness),  combined with Lemma \ref{lem:Killing}  shows that $\bar{\Lambda}$ is a holomorphic vector field. 
\begin{lem}
\label{lem:holomorph}
Let $(M,g,J)$ be a K\"ahler manifold. Let $K$ be a vector field of the form $K=J\grad\,f$ for some function $f$. Then $K$ is a Killing vector field for $g$, if and only if $K$ is holomorphic.
\end{lem}
\begin{proof}  We use that 
  $\nabla J=0$ and  that  $\nabla\grad\,f$ is a self-adjoint $(1,1)$-tensor. By direct calculation, we obtain 
\begin{align}
g(Y,(\mathcal{L}_{K}J)X)&=g(Y,J\nabla_{X}K)-g(Y,\nabla_{JX}K)=-g(Y,\nabla_{X}\grad\,f)-g(Y,\nabla_{JX}K)\nonumber\\
&=-g(X,\nabla_{Y}\grad\,f)-g(Y,\nabla_{JX}K)=-g(JX,\nabla_{Y}K)-g(\nabla_{JX}K,Y)\nonumber
\end{align} 
for arbitrary vectors $X$ and $Y$. It follows that $\mathcal{L}_{K}J=0$, if and only if $K$ satisfies the Killing equation as we claimed.
\end{proof}
\begin{cor}[\cite{ApostolovI}]
\label{cor:com}
Let $(M,g,J)$ be a K\"ahler manifold of real dimension $2n\geq 4$.  
Then, for every $A\in\mbox{Sol}(g)$ the vector fields $\Lambda$ and $\bar{\Lambda}$ from  (\ref{f}) are holomorphic and commuting,  i.e.,
\begin{align}
\mathcal{L}_{\Lambda}J=\mathcal{L}_{\bar{\Lambda}}J=0\mbox{ and }[\Lambda,\bar{\Lambda}]=0.\nonumber
\end{align} 
\end{cor}
\begin{proof}
By Remark \ref{rem:lambda2}, $\Lambda$ is the gradient of a function.  Since $\bar{\Lambda}=J\Lambda$ is a Killing vector field, by  Lemma \ref{lem:holomorph} we have that  $\bar{\Lambda}$ is holomorphic. Since  the  multiplication with the complex structure  sends  holomorphic vector fields to  holomorphic vector fields,    $\Lambda$ is holomorphic as well. By direct  calculations,  
$
[\Lambda,\bar{\Lambda}]=(\mathcal{L}_{\Lambda}J)\Lambda+J[\Lambda,\Lambda]=0.\nonumber
$
\end{proof}

\subsection{Covariant derivatives  of the eigenvectors of $A$}
\label{sec:eigenspaces}
Let $A$ be a complex, self-adjoint solution of equation (\ref{f}). On $M^0$, the eigenspace distributions $E_{A}(\mu_i)$  are well-defined  and differentiable. In general, they are 
 not integrable (except  for  the trivial  case when the metrics are affinely equivalent). The next proposition explains the behavior of these  distributions; it is essentially equivalent to  \cite[Proposition 14 and equation $(62)$]{ApostolovI}.
\begin{prop}
\label{prop:modulo}
Let $(M,g,J)$ be a Riemannian K\"ahler manifold and assume  $A\in\mbox{Sol}(g)$. Let $U$ be a smooth field of   eigenvectors  of $A$  defined on some open subset of $M^{0}$.  Let  $\rho$ be the corresponding eigenvalue. Then, for an arbitrary vector $X\in TM$, we have 
\begin{align}
(A-\rho \Id)\nabla_{X}U=X(\rho)U-g(U,X)\Lambda-g(U,\Lambda)X-g(U,JX)\bar{\Lambda}-g(U,\bar{\Lambda})JX.\label{eq:modulo}
\end{align}
Moreover, if  $V$   is an  eigenvector  of $A$  corresponding to an eigenvalue  $\tau\neq\rho$, then $V(\rho)=0$  and $grad \, \rho \in E_{A}(\rho)$. 
\end{prop}
\begin{proof}
Using equation (\ref{f}), we obtain
\begin{align}
(\nabla_{X}A)U=g(U,X)\Lambda+g(U,\Lambda)X+g(U,JX)\bar{\Lambda}+g(U,\bar{\Lambda})JX\nonumber
\end{align}
for arbitrary $X\in TM$. On the other hand, since $U\in E_{A}(\rho)$, we calculate
\begin{align}
\nabla_{X}(AU)=\nabla_{X}(\rho U)=X(\rho)U+\rho \nabla_{X}U.\nonumber
\end{align}
Inserting the last two equations in $\nabla_{X}(AU)=(\nabla_{X}A)U+A(\nabla_{X}U)$, we obtain \eqref{eq:modulo}.\\
Now let $\tau$ be another eigenvalue of $A$, such that $\rho\neq\tau$, and let $V\in E_{A}(\tau)$. Replacing $V$ with $X$ in equation \eqref{eq:modulo} and using that $E_{A}(\rho)\perp E_{A}(\tau)$, we obtain
\begin{align}
(A-\rho \Id)\nabla_{V}U=V(\rho)U-g(U,\Lambda)V-g(U,\bar{\Lambda})JV.\nonumber
\end{align}
Since the left-hand side of the equation above is orthogonal  to $E_{A}(\rho)$, we immediately obtain $0=V(\rho)=g(V,\grad\,\rho)$. Thus, $\grad\,\rho$   is orthogonal to all eigenvectors corresponding to eigenvalues different from $\rho$ implying it lies in $E_{A}(\rho)$ as we claimed.
\end{proof}

\section{K\"ahler manifolds of degree of mobility $D(g)=2$ admitting essential $h$-projective vector fields}
\label{sec:D=2}
For closed manifolds, the condition $\mbox{HProj}_0 \ne \mbox{Iso}_0$ is equivalent to the existence of an essential (i.e., not affine) $h$-projective vector field.  The goal of this section is to prove the following
\begin{thm}
\label{thm:system}
Let $(M,g,J)$ be a closed, connected Riemannian  K\"ahler manifold of real dimension $2n\geq 4$ and of degree of mobility $D(g)=2$ admitting an essential $h$-projective vector field. 
Let $A \in Sol(g)$ with the corresponding vector field $\Lambda$. 

Then, almost every point $y\in M$ has a neighborhood $U(y)$ such that there exists a constant $B<0$ and   a smooth  function $\mu:U(y)\to \mathbb{R}$   such that the  system
\begin{align}
\begin{array}{c}
(\nabla_{X}A)Y=g(Y,X)\Lambda + g(Y,\Lambda)X + g(Y,JX)\bar{\Lambda}+g(Y,\bar{\Lambda})JX\vspace{2mm}\\
\nabla_{X}\Lambda=\mu X+B A(X)\vspace{2mm}\\
\nabla_{X}\mu=2Bg(X,\Lambda)
\end{array}\label{eq:system}
\end{align}
is satisfied for all $x$ in $U(y)$ and all $X,Y\in T_xU$.
\end{thm}

One should understand \eqref{eq:system} as the system of PDEs on the components of $(A, \Lambda, \mu)$.  
Actually, in the system \eqref{eq:system}, the first equation is the equation \eqref{f} and 
 is fulfilled by the definition of $\mbox{Sol}(g)$, so our goal is to prove the local 
 existence of  $B$ and $\mu$ such that the second and the third equation of \eqref{eq:system} are fulfilled. 

\begin{rem} If $D(g)\geq 3$, the conclusion of this theorem is still true if we allow all, i.e., not  necessary negative,  values of $B$. In this case we even do not need 
the  `closedness' assumptions (i.e., the statement is local) and the existence of an $h$-projective vector field, see \cite{FKMR}.  Theorem \ref{thm:system} essentially needs the existence of an $h$-projective vector field and is not true locally. 
\end{rem}

\subsection{The tensor $A$ has at most two constant and precisely  one non-constant eigenvalue}
First let us prove
\label{sec:onenonconstant}
\begin{lem}
\label{lem:htrafodegree}
Let $(M,g,J)$ be a K\"ahler manifold of real dimension $2n\ge 4$ and of degree of mobility $D(g)=2$. Suppose $f:M\rightarrow M$ is an $h$-projective transformation for $g$ and let $A$ be an element of $\mbox{Sol}(g)$. Then $f$ maps the set $M^{0}$ of typical points of $A$  onto $M^{0}$.
\end{lem}
\begin{proof}
Let $x$ be a point of $M^{0}$. Since the characteristic polynomial of $(f^{*}A)_{|x}$ is the same as for $A_{|f(x)}$, we have to show that the number of different eigenvalues of $(f^{*}A)_{|x}$ and $A_{|x}$ coincide. If $A$ is proportional to the identity on $TM$, the assertion follows immediately. Let us therefore assume that $\{A,\Id\}$ is a basis for $\mbox{Sol}(g)$. We can find neighborhoods $U_{x}$ and $f(U_{x})$ of $x$ and $f(x)$ respectively, such that $A$ is non-degenerate in these neighborhoods (otherwise we add $t \cdot \Id$ to $A$ with a sufficiently large $t\in \mathbb{R}_+$). By \eqref{inverse},  $\bar{g}=(\det\,A)^{-\frac{1}{2}}g\circ A^{-1},g,f^{*}g$ and $f^{*}\bar{g}$ are $h$-projectively equivalent to each other in  $U_{x}$. By direct calculation, we see  
 that $f^{*}A=f^{*}A(g,\bar{g})=A(f^{*}g,f^{*}\bar{g})$.  Hence, $f^{*}A$ is contained in $\mbox{Sol}(f^{*}g)$. First suppose that $A(g,f^{*}g)$ is proportional to the identity. We obtain that
\begin{align}
f^{*}A=\alpha A+\beta \Id\nonumber
\end{align}
for some constants $\alpha,\beta$. Since $\alpha\neq 0$ (if $A$ is non-proportional to $\Id$, the same holds for $f^{*}A$), the number of different eigenvalues of $(f^{*}A)_{|x}$ is the same as for $A_{|x}$. It follows that $f(x)\in M^{0}$. Now suppose that $A(g,f^{*}g)$ is non-proportional to $\Id$. Then, 
 the numbers of different eigenvalues for $A_{|x}$ and $A(g,f^{*}g)_{|x}$ coincide. By Lemma \ref{lem:degree}, $D(f^{*}g)=2$ and $\{A(g,f^{*}g)^{-1},\Id\}$ is a basis for $\mbox{Sol}(f^{*}g)$. We obtain that
\begin{align}
f^{*}A=\gamma A(g,f^{*}g)^{-1}+\delta \Id\nonumber
\end{align}
for some constants $\gamma\neq0$ and $\delta$. It follows that the numbers of different eigenvalues of $(f^{*}A)_{|x}$ and $A(g,f^{*}g)^{-1}_{|x}$ coincide. Thus, the number of different eigenvalues of $(f^{*}A)_{|x}$ is equal to the number of different eigenvalues of $A_{|x}$. Again we have that  $f(x)\in M^{0}$ as we claimed.
\end{proof}

\begin{con}
\label{con:D=2}
In what follows, $(M,g,J)$ is a closed, connected Riemannian K\"ahler manifold of real dimension $2n\geq 4$ and of degree of mobility $D(g)=2$. We assume that  $v$ is  an 
$h$-projective vector field which is not  affine. 
 We chose a real number $t_{0}$ such that the pullback $\bar{g}:=(\Phi^{v}_{t_{0}})^{*}g$ \weg{is $h$-projectively equivalent to $g$, but} is  not  affinely equivalent to $g$. Let $A=A(g,\bar{g})$ be the corresponding element in $\mbox{Sol}(g)$ constructed by formula \eqref{eq:a}.
\end{con}

\begin{lem}
\label{lem:liederivA}
The tensor $A$ and the $h$-projective vector field $v$ satisfy
\begin{align}
\mathcal{L}_{v}A=c_{2}A^{2}+c_{1}A+c_{0}\Id\label{eq:liederivA}
\end{align}
for some constants $c_{2}\neq 0,c_{1},c_{0}$. 
\end{lem}   
\begin{proof}
 Note that the vector field $v$ is also $h$-projective with respect to the metric $\bar{g}$ and the degrees of mobility of the metrics $g$ and $\bar{g}$ are both equal to two (see Lemma \ref{lem:degree}). Since $A=A(g,\bar{g})$ is not proportional to the identity and $A(\bar{g},g)=A(g,\bar{g})^{-1}\in\mbox{Sol}(\bar{g})$, we obtain that $\{A,\Id\}$ and $\{A^{-1},\Id\}$ form bases for $\mbox{Sol}(g)$ and $\mbox{Sol}(\bar{g})$ respectively. It follows from Lemma \ref{lem:htrafo} that
\begin{align}
\begin{array}{l}
g^{-1}\circ\mathcal{L}_{v}g-\frac{\trace(g^{-1}\circ\mathcal{L}_{v}g)}{2(n+1)}\Id=\beta_{1}A+\beta_{2}\Id,\vspace{2mm}\\
\bar{g}^{-1}\circ\mathcal{L}_{v}\bar{g}-\frac{\trace(\bar{g}^{-1}\circ\mathcal{L}_{v}\bar{g})}{2(n+1)}\Id=\beta_{3}A^{-1}+\beta_{4}\Id
\end{array}\label{eq:liederivA4}
\end{align}
for some constants $\beta_{1},\beta_{2},\beta_{3}$ and $\beta_{4}$. Taking the trace on both sides of the above equations, we see that they are equivalent to
\begin{align}
\begin{array}{l}
g^{-1}\circ\mathcal{L}_{v}g=\beta_{1}A+\left(\frac{1}{2}\beta_{1}\,\trace\,A+(n+1)\beta_{2}\right)\Id,\vspace{2mm}\\
\bar{g}^{-1}\circ\mathcal{L}_{v}\bar{g}=\beta_{3}A^{-1}+\left(\frac{1}{2}\beta_{3}\,\trace\,A^{-1}+(n+1)\beta_{4}\right)\Id.
\end{array}\label{eq:liederivA1}
\end{align}
By  \eqref{inverse},  $\bar{g}$ can be written as $\bar{g}=(\det\,A)^{-\frac{1}{2}}g\circ A^{-1}$. Then, 
\begin{align}
\bar{g}^{-1}\circ\mathcal{L}_{v}\bar{g}&\stackrel{\eqref{inverse}}{=}(\det\,A)^{\frac{1}{2}}A\circ g^{-1}\circ\mathcal{L}_{v}((\det\,A)^{-\frac{1}{2}}g\circ A^{-1})\nonumber\\
&=-\frac{1}{2}(\det\,A)^{-1}(\mathcal{L}_{v}\det\,A)\Id+A\circ(g^{-1}\circ\mathcal{L}_{v}g)\circ A^{-1}-(\mathcal{L}_{v}A)\circ A^{-1}.\nonumber
\end{align}
We insert the second equation of \eqref{eq:liederivA1} in the left-hand side, the first equation of \eqref{eq:liederivA1} in the right-hand side and multiply with $A$ from the right to obtain
$$\begin{array}{rr} &
\beta_{3}\Id+\left(\frac{1}{2}\beta_{3}\,\trace\,A^{-1}+(n+1)\beta_{4}\right)A\vspace{1mm}\\ =&-\frac{1}{2}(\det\,A)^{-1}(\mathcal{L}_{v}\det\,A)A+\beta_{1}A^{2}+\left(\frac{1}{2}\beta_{1}\,\trace\,A+(n+1)\beta_{2}\right)A-\mathcal{L}_{v}A.
\end{array}$$
Rearranging the terms in the last equation, we obtain 
\begin{align}
\mathcal{L}_{v}A=c_{2}A^{2}+c_{1}A+c_{0}\Id\label{eq:liederivA2}
\end{align}
for constants  $c_{2}=\beta_{1}$, $c_0= -\beta_3$,   and a certain  function $c_{1}$.

\begin{rem} Our way to obtain the equation  \eqref{eq:liederivA2} is based on an idea of Fubini from 
\cite{Fubini1} used in the theory of projective vector fields. 
\end{rem} 

Our next goal is  to show that $c_{2}=\beta_{1}\neq 0$. If $\beta_{1}=0$, the first equation of \eqref{eq:liederivA1} reads
\begin{align}
\mathcal{L}_{v}g=(n+1)\beta  g\nonumber
\end{align}
hence, $v$ is an infinitesimal homothety for $g$. This contradicts the assumption that $v$ is essential and we obtain that $c_{2}=\beta_{1}\neq 0$.\\
Now let us show that the function $c_{1}$ is a constant. Since $A$ is nondegenerate,  $c_1$ is a smooth function, so it is sufficient to show that its differential vanishes at every point of $M^0$. 
We will  work in a neighborhood of a point of $M^{0}$. Let $U\in E_{A}(\rho)$ be an eigenvector of $A$ with corresponding eigenvalue $\rho$. Using the Leibniz rule for the Lie derivative and the condition that $U\in E_{A}(\rho)$,  we obtain the equations
\begin{align}
\mathcal{L}_{v}(AU)=\mathcal{L}_{v}(\rho U)=v(\rho)U+\rho [v,U]\mbox{  and  }\mathcal{L}_{v}(AU)=(\mathcal{L}_{v}A)U+A([v,U]).\nonumber
\end{align}
Combining both equations and inserting $\mathcal{L}_{v}A$ from \eqref{eq:liederivA2}, we obtain
\begin{align}
(v(\rho)-c_{2}\rho^{2}-c_{1}\rho-c_{0})U=(A-\rho \Id) [v,U].\nonumber
\end{align}
In a basis of eigenvectors $\{U_i, JU_i\}$ of $A$ from the proof of  Theorem  \ref{thm:Mat}, we see that the right-hand side does not contain any component from   $E_{A}(\rho)$ (i.e., the right-hand side is a linear combination of  eigenvectors corresponding to other eigenvalues). Then,  
\begin{align}
c_{1}=v(ln(\rho))-c_{2}\rho-\frac{c_{0}}{\rho}\mbox{  and  }(A-\rho \Id)[v,U]=0.\label{eq:liederivA3}
\end{align}
These equations are true for all eigenvalues $\rho$ of $A$ and corresponding eigenvectors $U$. Note that $\rho\neq 0$ since $A$ is non-degenerate. By construction, the metric $\bar{g}$ (such that $A=A(g,\bar{g})$) is not affinely equivalent to $g$, in particular, $A$ has more than one eigenvalue. Let be $W\in E_{A}(\mu)$ and $\rho\neq\mu$. Applying $W$ to the first equation in \eqref{eq:liederivA3} and using Proposition \ref{prop:modulo}, we obtain
\begin{align}
W(c_{1})=[W,v](ln(\rho)).\nonumber
\end{align}
The second equation of \eqref{eq:liederivA3} shows that $[v,W]=0\mbox{ modulo }E_{A}(\mu)$. Hence,
\begin{align}
W(c_{1})=0.\nonumber
\end{align}
We obtain that $U(c_{1})=0$ for all eigenvectors $U$ of $A$.  Then, $dc_1\equiv 0$ on $M^0$.  Since $M^{0}$ is dense in $M$, we  obtain that $dc_1\equiv 0$ on the whole $M$  implying  $c_{1}$ is a constant. This completes the proof of Lemma \ref{lem:liederivA}.
\end{proof}
\begin{con} Since $c_{2}\neq 0$, we can replace $v$ by  the $h$-projective vector field $\frac{1}{c_{2}}v$. For simplicity, we denote the new vector field again by  $v$; this implies that equation \eqref{eq:liederivA} is now satisfied for $c_{2}=1$:  instead of \eqref{eq:liederivA} we have 
\begin{align}
\mathcal{L}_{v}A=A^{2}+c_{1}A+c_{0}\Id\label{eq:lie}
\end{align}
for some constants $c_{1},c_{0}$. 
\end{con}
\begin{rem}
\label{rem:Av}
Note that the constant $\beta_{1}$ in the proof of Lemma \ref{lem:liederivA} is equal to $c_{2}$. With the convention above, the first equation in \eqref{eq:liederivA4} now reads
\begin{align}
A_{v}=g^{-1}\circ\mathcal{L}_{v}g-\frac{\trace(g^{-1}\circ\mathcal{L}_{v}g)}{2(n+1)}\Id=A+\beta\Id\label{eq:Av}
\end{align}
for some $\beta\in\mathbb{R}$.
\end{rem}

\begin{rem}
In the proof of Lemma \ref{lem:liederivA}, we had to do some additional work to show that $c_{1}$ is indeed a constant. This problem does not appear if we use the $h$-projectively invariant formulation of equation \eqref{f}. We introduce this approach in Appendix \ref{app} where we also give an alternative proof of Lemma \ref{lem:liederivA}.
\end{rem}

\begin{lem}
\label{lem:spectrumD=2}
The tensor $A$ has precisely one non-constant eigenvalue $\rho$ of multiplicity $2$ and at least one and at 
most two constant eigenvalues (we denote the constant eigenvalues by $\rho_{1}<\rho_{2}$ and their multiplicities by $2k_{1}$ and $2k_{2}=2n-2k_{1}-2$ respectively; we allow $k_1$ to be equal to $0$ and $n-1$; if $k_1=0$, $A$ has only one constant eigenvalue $\rho_2$ and if $k_1=n-1$, than $A$ has 
only one constant eigenvalue $\rho_1$).    Moreover, the eigenvalues satisfy the equations 
\begin{align}
\begin{array}{c}
0=\rho_{1}^{2}+c_{1}\rho_{1}+c_{0}=\rho_{2}^{2}+c_{1}\rho_{2}+c_{0}\vspace{2mm}\\
v(\rho)=\rho^{2}+c_{1}\rho+c_{0}\end{array}\label{eq:vrho}
\end{align}
for the constants $c_{1},c_{0}$ from \eqref{eq:lie}. For every  point $x\in M^{0}$ such that $d\rho_{|x}\neq 0$ and $v(x)\neq0$, the evolution of the non-constant eigenvalue $\rho$ along the flow line $\Phi^{v}_{t}(x)$ is given by 
\begin{align}
\rho(t)=-\frac{c_{1}}{2}-\sqrt{\alpha}\,\tanh(\sqrt{\alpha}(t+d))\label{eq:odesol}, 
\end{align}
 where $\alpha=\tfrac{1}{4}c_{1}^{2}-c_{0}$ is necessarily a positive real number.\\
\end{lem}
\begin{figure}
  \includegraphics[width=.5\textwidth]{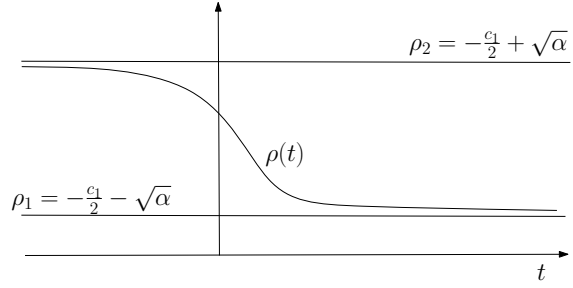}
  \caption{The behavior of the restriction of the eigenvalues to the integral curve of $v$: at most 
  two eigenvalues, $\rho_1$ and $\rho_2$,  
   are constant; they are roots  of the quadratic polynomial $X^{2}+c_{1}X+c_{0}$. Precisely one eigenvalue, $\rho$,   is not constant along the integral curve  and is given by   \eqref{eq:odesol}.  }\label{picrho}
\end{figure}

\begin{proof}
We proceed as in the proof of Lemma \ref{lem:liederivA}. Applying equation \eqref{eq:lie} to an eigenvector $U$ of $A$, corresponding to the eigenvalue $\rho$ yields
\begin{align}
(\rho^{2}+c_{1}\rho+c_{0}-v(\rho))U=-(A-\rho \Id)\lbrack v,U\rbrack.\nonumber
\end{align}
Since the right-hand side does not contain any components lying in $E_{A}(\rho)$, we obtain that 
\begin{align}
(A-\rho \Id)\lbrack v,U\rbrack=0\mbox{ and }v(\rho)=\rho^{2}+c_{1}\rho+c_{0}\label{eq:liebracket}
\end{align}
for all eigenvalues $\rho$ of $A$ and all eigenvectors $U\in E_{A}(\rho)$.

In particular, each constant eigenvalue is a solution of the equation $\rho^{2}+c_{1}\rho+c_{0}=0$. This implies that there are at most two different constant eigenvalues $\rho_{1}$ and $\rho_{2}$ for $A$ as  we claimed.

On the other hand, let $\rho$ be a non-constant eigenvalue of $A$ (there is always a non-constant eigenvalue since otherwise, the vector field $\Lambda$ vanishes identically on $M$ and therefore, the metrics $g$ and $\bar{g}$ (such that $A=A(g,\bar{g})$) are already affinely equivalent, see Remark \ref{rem:lambda}) and let $x\in M^{0}$ be a point such that $d\rho_{|x}\neq 0$ and $v(x)\neq 0$. The second equation in \eqref{eq:liebracket} shows that the restriction of $\rho$ to the flow line $\Phi^{v}_{t}(x)$ of $v$ (i.e., $\rho(t):=\rho( \Phi^{v}_{t}(x)$) satisfies the ordinary differential equation 
\begin{align}
\dot\rho=\rho^{2}+c_{1}\rho+c_{0}, \ \ \textrm{where $\dot\rho$ stays for $\tfrac{d}{dt}\rho$} .\label{eq:ode}
\end{align}
This ODE can be solved explicitly; the solution (depending on the parameters $c_1$, $c_0$) is given by the following list. We put    $\alpha=\frac{c_{1}^{2}}{4}-c_{0}$.
\begin{itemize}
\item For $\alpha<0$, the non-constant solutions of equation \eqref{eq:ode} are of the form
\begin{align}
-\frac{c_{1}}{2}-\sqrt{-\alpha}\,\mbox{tan}(\sqrt{-\alpha}(-t+d)).\nonumber
\end{align}
\item For $\alpha>0$, the non-constant solutions of equation \eqref{eq:ode} take the form
\begin{align}
-\frac{c_{1}}{2}-\sqrt{\alpha}\,\mbox{tanh}(\sqrt{\alpha}(t+d))\mbox{ or }-\frac{c_{1}}{2}-\sqrt{\alpha}\,\mbox{coth}(\sqrt{\alpha}(t+d)).\nonumber
\end{align}
\item For $\alpha=0$, the non-constant solutions of equation \eqref{eq:ode} are given by
\begin{align}
-\frac{c_{1}}{2}-\frac{1}{t+d}.\nonumber
\end{align}
\end{itemize}
Since the degree of mobility is equal to $2$, we can apply Lemma \ref{lem:htrafodegree} to obtain that the flow $\Phi^{v}_{t}$ maps $M^{0}$ onto $M^{0}$. It follows that $\rho(t)$ satisfies equation \eqref{eq:ode} for all $t\in\mathbb{R}$. However, the only solution of \eqref{eq:ode} which does not reach infinity in finite time is 
\begin{align}
-\frac{c_{1}}{2}-\sqrt{\alpha}\,\mbox{tanh}(\sqrt{\alpha}(t+d))\nonumber,
\end{align}
 where $\alpha=\frac{c_{1}^{2}}{4}-c_{0}$ is necessarily a positive real number.\\
We obtain that the non-constant eigenvalues of $A$ satisfy equation \eqref{eq:odesol}, in particular, their images contain the open interval $(-\frac{c_{1}}{2}-\sqrt{\alpha},-\frac{c_{1}}{2}+\sqrt{\alpha})$. Suppose that there are two different non-constant eigenvalues $\rho= -\tfrac{c_{1}}{2}-\sqrt{\alpha}\,\mbox{tanh}(\sqrt{\alpha}(t+d))$ and $\tilde\rho=-\frac{c_{1}}{2}-\sqrt{\alpha}\,\mbox{tanh}(\sqrt{\alpha}(t+\tilde d))$ of $A$. Then we can find points $x_{0},x_{1},x_{2}\in M$ such that $\rho(x_{0})<\tilde \rho(x_{1})<\rho(x_{2})$. This contradicts the global ordering of the eigenvalues of $A$, see Theorem \ref{thm:Mat}(1). It follows that $A$ has precisely  one non-constant eigenvalue $\rho$. This eigenvalue restricted to flow lines of $v$ satisfies equation \eqref{eq:odesol}. By Corollary \ref{cor:mult}, the multiplicity of $\rho$ is equal to two. We obtain that there must be at least one constant eigenvalue of $A$. Finally, Lemma \ref{lem:spectrumD=2} is proven.
\end{proof}

\begin{cor} In the notation above, all eigenvalues $\rho_1, \rho, \rho_2$ are smooth functions on the manifold.    
\end{cor}
\begin{proof}  The eigenvalues $\rho_1, \rho_2$ are constant and are therefore smooth. The non-constant eigenvalue $\rho$ is equal to $\tfrac{1}{2}\trace\,{A} - k_1 \rho_1 - (n-1-k_1) \rho_2$ and is therefore also smooth.
\end{proof}

\begin{lem}
\label{lem:onenonconstant}
Let $A$  have only one non-constant  eigenvalue denoted by $\rho$.
 On $M_{d\rho\neq0}:=\{x\in M:d\rho_{|x}\neq 0\}$, the vector fields $\Lambda$ and $\bar{\Lambda}$ are eigenvectors of $A$ corresponding to the eigenvalue $\rho$,  i.e., $E_{A}(\rho)=\spann\{\Lambda,\bar{\Lambda}\}$.\\
Moreover, $M_{d\rho\neq0}$ is open and dense in $M$ and $\Lambda(\rho)\neq 0$ on $M_{d\rho\neq0}$.  
\end{lem}

\begin{rem}
Note that the second part of the assertion above is still true even locally and even if there are more than just one non-constant eigenvalue. The proof is based on the existence of a  family of Killing vector fields (one for each non-constant eigenvalue) and is given in \cite[Proposition 14]{ApostolovI}. 
\end{rem}

\begin{proof}
First of all, since $\rho$ is the only non-constant eigenvalue of $A$ and $\rho$ has multiplicity equal to $2$ (see Corollary \ref{cor:mult}), we obtain $\Lambda=\frac{1}{4}\grad\,\trace\,A=\frac{1}{2}\grad\,\rho$.

By Proposition \ref{prop:modulo}, \, $\Lambda$ is an eigenvector of $A$ corresponding to the eigenvalue $\rho$. Since the eigenspaces of $A$ are invariant with respect to the complex structure $J$, we immediately obtain $E_{A}(\rho)=\spann\{\Lambda,\bar{\Lambda}\}$. Moreover, since 
 $\grad\,\rho$ is proportional to $\Lambda$,  we have $\bar{\Lambda}(\rho)=0$ and  $\Lambda(\rho)\neq 0$ on $M_{d\rho\neq0}$.
 
Obviously, $M_{d\rho\neq0}$ is an open subset of $M$. As we explained above,   $d\rho_{|x}=0$, if and only if $\Lambda(x)=\bar{\Lambda}(x)=0$. Then, $M\setminus M_{d\rho\neq0}$ coincides with the set of zeros of the non-trivial Killing vector field $\bar{\Lambda}$. We obtain that $M_{d\rho\neq0}$ is dense in $M$.
\end{proof}
Let us now consider the critical points of the non-constant eigenvalue $\rho$:
\begin{lem}
\label{lem:minmax}
At every $x$ such that  $d\rho_{|x}=0$,  $\rho$ takes its maximum or minimum values $\rho=-\frac{c_{1}}{2}\pm\sqrt{\alpha}$, where $\alpha=\frac{c_{1}^{2}}{4}-c_{0}$ and $c_{1},c_{0}$ are the constants from the  equation \eqref{eq:lie}. Moreover, $v\neq 0$ on $M_{d\rho\neq0}$. 
\end{lem}
\begin{proof}
Since the subsets $M_{v\neq0}$ and $M_{d\rho\neq0}$ are both open and dense in $M$ (see Corollary \ref{cor:htrafo2} and Lemma \ref{lem:onenonconstant}), we obtain that $M^{1}=M_{v\neq0}\cap M_{d\rho\neq0}$ is open and dense in $M$ as well. Equation \eqref{eq:odesol} shows that $-\frac{c_{1}}{2}-\sqrt{\alpha}<\rho(x)< -\frac{c_{1}}{2}+\sqrt{\alpha}$ for all $x\in M^{1}$. Since $M^{1}$ is dense, we obtain
\begin{align}
-\frac{c_{1}}{2}-\sqrt{\alpha}\leq\rho(x)\leq -\frac{c_{1}}{2}+\sqrt{\alpha}\nonumber
\end{align}
for all $x\in M$. Now suppose that $d\rho_{|x}=0$ for some $x\in M$. It follows from equation \eqref{eq:vrho} that $\rho(x)$ satisfies $0=d\rho_{|x}(v)=\rho(x)^{2}+c_{1}\rho(x)+c_{0}$, hence, $\rho(x)$ is equal to the maximum or minimum value of $\rho$. Now suppose $v(x)=0$. By  \eqref{eq:vrho},   $\rho$ takes its maximum or minimum value at $x$. It follows that $d\rho_{|x}=0$.
\end{proof}

\subsection{Metric components on integral manifolds of $\spann\{\Lambda,\bar{\Lambda}\}$}
\label{sec:metricinadaptedcoord}
By  Lemma \ref{lem:spectrumD=2},  $A$ has precisely  one non-constant eigenvalue $\rho$ and at most two constant eigenvalues $\rho_{1}$ and $\rho_{2}$.
The goal of this section is to calculate the components of the restriction of the metric $g$ to the integral manifolds of the eigenspace distribution $E_{A}(\rho)=\spann\{\Lambda,\bar{\Lambda}\}$. In order to do it, we split the tangent bundle on $M_{d\rho\neq0}$ into the direct product of   two distributions: 
\begin{align}
D_{1}:=\spann\{\Lambda\}\mbox{ and }D_{2}:=D_{1}^{\perp}=\spann\{\bar{\Lambda}\}\oplus E_{A}(\rho_{1})\oplus E_{A}(\rho_{2})\nonumber
\end{align}
 First let us show
\begin{lem}
\label{lem:integrable}
The distributions $D_{1}$, $D_{2}$ and $E_{A}(\rho)$ are integrable on $M_{d\rho\neq0}$. Moreover, integral manifolds of $D_{1}$ and $E_{A}(\rho)$ are totally geodesic. 
\end{lem}
\begin{proof}
Since $\Lambda$ is a gradient, the distribution $D_{2}$ is integrable. On the other hand, Corollary \ref{cor:com} immediately implies  that $E_{A}(\rho)$ is integrable. The distribution $D_1$ is one-dimensional and is therefore integrable. In order to show that the integral manifolds of $D_{1}$ and $E_{A}(\rho)$ are totally geodesic, we  consider  the (quadratic in velocities)  integrals $I_{0},I_{1},I_{2}:TM\rightarrow \mathbb{R}$ given by
\begin{align}
\begin{array}{c}
I_{0}(\zeta)=g(\bar{\Lambda},\zeta)^{2},I_{1}(\zeta)=\left(\frac{d^{k_{1}-1}}{dt^{k_{1}-1}}F_{t}(\zeta)\right)|_{t=\rho_{1}}\mbox{ and }I_{2}(\zeta)=\left(\frac{d^{k_{2}-1}}{dt^{k_{2}-1}}F_{t}(\zeta)\right)|_{t=\rho_{2}}
\end{array}\label{eq:int3},
\end{align}
 where $2k_{1},2k_{2}$ are the multiplicities of the constant eigenvalues $\rho_{1},\rho_{2}$ of $A$.\\
If $s:TM\rightarrow \mathbb{R}$ is a quadratic polynomial in the velocities, we define the {\it nullity} 
 of $s$ by  $$\mbox{\bf null}\,s:=\{\zeta\in TM:s(\zeta)=0\}.$$ In the   orthonormal frame of eigenvectors of $A$ from the proof of Theorem \ref{thm:Mat},  the integrals $F_t$ are given by \eqref{eq:int2expl}, 
  and  it is easy to see that $$\mbox{\bf null}I_{1}=E_{A}(\rho)\oplus E_{A}(\rho_{2}),\  \mbox{\bf null}I_{2}=E_{A}(\rho)\oplus E_{A}(\rho_{1}) \textrm{ and } \mbox{\bf null}I_{0}=\spann\{\Lambda\}\oplus E_{A}(\rho_{1})\oplus E_{A}(\rho_{2}).$$ It follows that $D_{1}=\mbox{\bf null}I_{0}\cap\mbox{\bf null}I_{1}\cap\mbox{\bf null}I_{2}$ and $E_{A}(\rho)=\mbox{\bf null}I_{1}\cap\mbox{\bf null}I_{2}$.  Since the functions are integrals, if  $ \dot\gamma(0) \in \mbox{\bf null}I_{i}$ , then   $\dot\gamma(t) \in \mbox{\bf null}I_{i}$ for all $t$. Then,  
 every  geodesic $\gamma$
   such that $\dot\gamma(0)\in D_{1}$ (resp. $E_{A}(\rho)$)  remains  tangent to $D_{1}$ (resp. $E_{A}(\rho)$). Thus,  the integral manifolds of $D_{1}$ and $E_{A}(\rho)$ are totally geodesic.
\end{proof}
Let us introduce local coordinates $x^{1},x^{2},...,x^{2n}$ in a neighborhood of a point of $M_{d\rho\neq0}$ such that (for all constants $C_{1},...,C_{2n}$)
 the equation $x^{1}=C_{1}$ defines an integral manifold of $D_{2}$ and  the system $\{x^{i}=C_i\}_{i=2,...,2n}$  defines an integral manifold of $D_{1}$. In these coordinates, the metric $g$  has  the block-diagonal form
\begin{align}
g=g_{11}dx^{1}\otimes dx^{1}+\sum_{i,j=2}^{2n}\tilde{g}_{ij}dx^{i}\otimes dx^{j}.\nonumber
\end{align}
In what follows we  call  such coordinates  \textit{adapted to the decomposition} $TM_{|M_{d\rho\neq0}}=D_{1}\oplus D_{2}$. Let us show that the $h$-projective vector field $v$ splits  into two independent components with respect to this decomposition:
\begin{lem}
\label{lem:decomp}
In the coordinates $x^{1},x^{2},...,x^{2n}$ adapted to the decomposition $TM_{|M_{d\rho\neq0}}=D_{1}\oplus D_{2}$, the $h$-projective vector field $v$ is given by
\begin{equation}\label{25} 
v=\underbrace{v^{1}(x^{1})\partial_{1}}_{=:v_{1}\in D_{1}}+\underbrace{v^{2}(x^{2},...,x^{2n})\partial_{2}+...+v^{2n}(x^{2},...,x^{2n})\partial_{2n}}_{=:v_{2}\in D_{2}}
\end{equation}
\end{lem}
\begin{proof}
Since $\bar{\Lambda}$ is an eigenvector of $A$ corresponding to the non-constant eigenvalue $\rho$, the first equation in \eqref{eq:liebracket} implies that
\begin{align}
\lbrack v,\bar{\Lambda}\rbrack=f\bar{\Lambda}+h\Lambda\nonumber
\end{align}
for some functions $f,h$.  If we apply $d\rho$ to both sides of the equation above, we obtain $\bar{\Lambda}(v(\rho))=\bar{\Lambda}(\rho^{2}+c_{1}\rho+c_{0})=0$ on the left-hand side and $h\Lambda(\rho)$ on the right-hand side. Since $\Lambda(\rho)\neq 0$ on $M_{d\rho\neq0}$, we necessarily have $h=0$. By definition $v$ is holomorphic and since $\bar{\Lambda}=J\Lambda$, we see that the equations 
\begin{align}
\lbrack v,\bar{\Lambda}\rbrack=f\bar{\Lambda}\mbox{ and }\lbrack v,\Lambda\rbrack=f\Lambda\label{eq:brack1}
\end{align}
are satisfied.\\
For an eigenvector $U$ of $A$, corresponding to some constant eigenvalue $\mu$, the first equation in \eqref{eq:liebracket} shows that
\begin{align}
\lbrack v,U\rbrack\in E_{A}(\mu).\label{eq:brack2}
\end{align}
For each index $i\geq 2$, $\partial_{i}$ is contained in $D_{2}$. On the other hand, $\partial_{1}$ is always proportional to $\Lambda$. We obtain 
\begin{align}
\partial_{i}\sim\bar{\Lambda}\mbox{ mod }E_{A}(\rho_{1})\oplus E_{A}(\rho_{2})\mbox{ and }\partial_{1}\sim\Lambda.\nonumber
\end{align}
Using equation \eqref{eq:brack1} and equation \eqref{eq:brack2}, we see that
\begin{align}
\lbrack v,\partial_{i}\rbrack\in D_{2}\mbox{ for all $i\geq 2$ and }\lbrack v,\partial_{1}\rbrack\in D_{1}.\nonumber
\end{align}
This means that $\partial_{i}v^{1}=0$ and $\partial_{1}v^{i}=0$ for all $i\geq 2$. Hence, 
\begin{align}
v=(v^{1}(x^{1}),v^{2}(x^{2},...,x^{2n}),...,v^{2n}(x^{2},...,x^{2n}))\nonumber
\end{align}
as we claimed.
\end{proof}
Let us write $v=v_{1}+v_{2}$ with respect to the decomposition $TM_{|M_{d\rho\neq0}}=D_{1}\oplus D_{2}$ (as in \eqref{25}). The vector fields $v_{1}$ and $v_{2}$ are well-defined and smooth on   $M_{d\rho\neq0}$.  By Lemma \ref{lem:decomp},  we have  $\lbrack v_{1},v_{2}\rbrack=0$.
\begin{lem}
\label{lem:propv1}
The non-constant eigenvalue $\rho$ satisfies the equation $v_{1}(\rho)=\rho^{2}+c_{1}\rho+c_{0}$ and the evolution of $\rho$ along the flow-lines of $v_{1}$ is given by equation \eqref{eq:odesol}. 
Moreover, $v_{1}$ is a non-vanishing complete vector field on $M_{d\rho\neq0}$. 
\end{lem}
\begin{proof}
Since by  Proposition \ref{prop:modulo} and Lemma \ref{lem:onenonconstant} we have  $d\rho(V)=0$ for all  $V\in D_{2}$, we  have $v_{2}(\rho)=0$ and, hence, $v_{1}(\rho)=v(\rho)=\rho^{2}+c_{1}\rho+c_{0}$. Using Lemma \ref{lem:decomp}, we obtain that the restriction of $\rho$ on the flow line $\Phi^{v_{1}}_{t}(x)$ coincides with the restriction of $\rho$ on $\Phi^{v}_{t}(x)$ for all $x\in M_{d\rho\neq0}$. Therefore the evolution of $\rho$ along flow lines of $v_{1}$ is again given by equation \eqref{eq:odesol}.\\
Let us assume that $v_{1}(x)=0$ for some point $x\in M_{d\rho\neq0}$. We obtain that $0=\rho(x)^{2}+c_{1}\rho(x)+c_{0}$, which implies that $\rho(x)$ is a maximum or minimum value of $\rho$ (see Lemma \ref{lem:minmax}). It follows that $d\rho_{|x}=0$, contracting our assumptions.\\
Finally, let us show that $v_{1}$ is complete. Take a maximal integral curve $\sigma:(a,b)\to M_{d\rho\ne 0}$ of $v_{1}$ and assume  $b<\infty$. Since $M$ is closed, there exists a   sequence $\{b_{n}\}\subset  (a,b)$, converging to  $b$  such that $\mbox{lim}_{\mbox{\tiny $n\rightarrow \infty$}}\sigma(b_{n})=y$ for some $y\in M$. Then,  $y\in M\setminus M_{d\rho\neq 0}$ since otherwise  
 the maximal interval $(a,b)$ of $\sigma$ can be extended beyond $b$. Then,  $d\rho_{|y}=0$,  and Lemma \ref{lem:minmax} implies that $\rho(y)$  is  equal to the minimum value $-\frac{c_{1}}{2}-\sqrt{\alpha}$. We obtain that $\mbox{lim}_{\mbox{\tiny $n\rightarrow \infty$}}\rho(\sigma(b_{n}))=-\frac{c_{1}}{2}-\sqrt{\alpha}$. On the other hand, formula \eqref{eq:odesol} shows that this value cannot be obtained in finite time $b<\infty$. This gives us  a contradiction  implying   $v_{1}$ is a complete vector field on $M_{d\rho\neq 0}$.
\end{proof}
Let us now  calculate the restriction of the metric $g$ to the integral manifolds of the distribution $E_{A}(\rho)=\spann\{v_{1},\bar{\Lambda}\}$. 
\begin{lem}
\label{lem:metricinadaptedcoord}
In a neighborhood of each point of $M_{d\rho\neq0}$, it is possible to choose the coordinates $t=x^{1},x^{2},...,x^{2n}$ adapted to the decomposition $TM_{|M_{d\rho\neq0}}=D_{1}\oplus D_{2}$ in such a way, that $v_{1}=\partial_{1}$, $\bar{\Lambda}=\partial_{2}$ and
\begin{align}
g=\left(\begin{array}{c|c|ccc}h&0&0&\dots&0\\
                \hline 0&g(\Lambda,\Lambda)&*&\dots&*\\
\hline 0&*&*&\dots&*\\
       \vdots&\vdots&\vdots&&\vdots\\
        0&*&*&\dots&*
      \end{array}\right).\label{eq:metric}
\end{align}
The functions $h=g(v_{1},v_{1}),g(\Lambda,\Lambda)$ and $\rho$ depend on the first coordinate $t$ only and are given explicitly by the formulas
\begin{align}
\begin{array}{l}
h(t)=D \frac{e^{\left(C- c_{1}\right)t}}{\cosh^{2}(\sqrt{\alpha}(t+d))},\vspace{4mm}\\
g(\Lambda,\Lambda)=\frac{\dot{\rho}^{2}}{4h}\mbox{ (where $\dot{\rho}=\tfrac{d\rho}{dt}$) and }\vspace{4mm}\\
\rho(t)=-\frac{c_{1}}{2}-\sqrt{\alpha}\,\tanh(\sqrt{\alpha}(t+d)).
\end{array}\label{eq:sol}
\end{align}
The constants $\alpha>0$ and $C$ in equation \eqref{eq:sol} are defined as $\alpha=\frac{c_{1}^{2}}{4}-c_{0}$ and $C=-\tfrac{n-1}{2}c_{1}-(2k_{1}+1-n)\sqrt{\alpha}+(n+1)\beta$, where $D>0,d,\beta,c_{1},c_{0}\in\mathbb{R}$ and $2k_{1}$ is the multiplicity of the constant eigenvalue $\rho_{1}$. The constants $c_{1},c_{0}$ are the same as in equation \eqref{eq:lie}. Moreover, $c_{1},c_{0}$ and $\beta$ are global constants,  i.e., they are the same for each coordinate system of the above type.
\end{lem}
\begin{proof}
In a neighborhood of an arbitrary point of $M_{d\rho\neq0}$, let us introduce a chart $x^{1},x^{2},...,x^{2n}$, adapted to the decomposition $TM_{|M_{d\rho\neq0}}=D_{1}\oplus D_{2}$. By Lemma \ref{lem:decomp} and Lemma \ref{lem:propv1}, we can choose  these coordinates such that the flow line parameter $t$ of $v_{1}$ coincides with $x^{1}$ (i.e., such that the first component of $v$ in the coordinate system equals $\tfrac{\partial}{\partial x^1}$). By \eqref{eq:brack1}, we have $[v,\bar{\Lambda}]\in D_{2}$. Moreover, $[v_{2},\bar{\Lambda}]\in D_{2}$ since $D_{2}$ is integrable. It follows that $[v_{1},\bar{\Lambda}]\in D_{2}$. On the other hand, since $v_{1}=f\Lambda$ for some function $f$ and $[\Lambda,\bar{\Lambda}]=0$, we obtain that $[v_{1},\bar{\Lambda}]=-\bar{\Lambda}(f)\Lambda\in D_{1}$, implying 
\begin{align}
[v_{1},\bar{\Lambda}]=0.\nonumber
\end{align}
It follows, that we can choose the second coordinate $x^{2}$ in such a way that $\bar{\Lambda}=\partial_{2}$.\\  
Next let us show that $h=g_{11}$ depends on the first coordinate of the adapted chart only. For this, let $I$ be an integral of second order for the geodesic flow of $g$ such that $I$ is block-diagonal with respect to the adapted coordinates $t,x^{2},...,x^{2n}$. For the moment we adopt the convention that latin indices run from $2$ to $2n$ such that $I$, considered as a polynomial on $T^{*}M$, can be written as $I=I^{11}p_{1}^{2}+I^{ij}p_{i}p_{j}$. We calculate the poisson bracket $0=\{H,I\}$ to obtain the equations
\begin{align}
0=I^{ik}\partial_{k}g^{11}-g^{ik}\partial_{k}I^{11}\mbox{ for all }i=2,...,2n.\label{eq:poisson}
\end{align}
Inserting integrals $I$ of special type, we can impose restrictions on the metric. Obviously the integrals $I_{0},I_{1},I_{2}$ defined in equation \eqref{eq:int3} are block-diagonal. On the other hand, in the proof of Lemma \ref{lem:integrable} it was shown that they satisfy $\mbox{\bf null}I_{1}=E_{A}(\rho)\oplus E_{A}(\rho_{2})$, $\mbox{\bf null}I_{2}=E_{A}(\rho)\oplus E_{A}(\rho_{1})$ and $\mbox{null}I_{0}=\spann\{\Lambda\}\oplus E_{A}(\rho_{1})\oplus E_{A}(\rho_{2})$. It follows that the integral $F=I_{0}+I_{1}+I_{2}$ is block-diagonal and that its nullity is  equal to $D_{1}$. Then $F$ can be written as $F^{ij}p_{i}p_{j}$ and the matrix $(F^{ij})_{i,j\geq 2}$ is invertible at each point where the coordinates are defined. Replacing the integral $I$ in equation \eqref{eq:poisson} with $F$ yields
\begin{align}
\partial_{i}g^{11}=0\nonumber
\end{align}
for all $2\leq i\leq 2n$ hence, the metric component $g_{11}=(g^{11})^{-1}$ depends on $t$ only.\\ 
Now let us show the explicit dependence of the functions $h,\rho$ and $g(\Lambda,\Lambda)$ on the parameter $t$. We already know that $h=g_{11}$ and $\rho$ depend on $t$ only (for $\rho$ this follows from Proposition \ref{prop:modulo} and Lemma \ref{lem:onenonconstant}) and by Lemma \ref{lem:propv1}, the dependence of $\rho$ on the first coordinate $t$ is given by equation \eqref{eq:odesol}.\\ 
Recall that $\lambda=\frac{1}{4}\trace\,A=\frac{1}{2}\rho+\mbox{const}$. It follows that $d\lambda=\frac{1}{2}\dot{\rho}\,dt$ and hence, $\Lambda=\grad\,\lambda=\tfrac{\dot{\rho}}{2h}\partial_{1}$. We obtain 
\begin{align}
g(\Lambda,\Lambda)=\frac{\dot{\rho}^{2}}{4h}.\nonumber
\end{align}
What is left is to clarify the dependence of the function $h$ on the parameter $t$. Note that in the coordinates $t,x^{2},...,x^{2n}$, the $h$-projective vector field $v$ is given by $v=\partial_{1}+v_{2}$. Let us denote by  $\dot{h}$ and $\dot{\rho}$ the derivatives of $h$ and $\rho$ with respect to the coordinate $t$ and denote the restriction of $g$ to the distribution $D_{2}$ by  $\tilde{g}$. Then we calculate
\begin{align}
\mathcal{L}_{v}g=\mathcal{L}_{v_{1}}g+\mathcal{L}_{v_{2}}g=\dot{h}\,dt\otimes dt+\mathcal{L}_{v_{1}}\tilde{g}+\mathcal{L}_{v_{2}}\tilde{g},\label{eq:liederiv}
\end{align}
where we used that $v_{2}(h)=0$ and $\mathcal{L}_{v_{2}}dt=0$ which follows from $[v_{1},v_{2}]=0$ and $[v_{2},\partial_{i}]\in D_{2}$ for all $i\geq 2$. Note that $\mathcal{L}_{v_{1}}\tilde{g}$ and $\mathcal{L}_{v_{2}}\tilde{g}$ do not contain any expressions involving $dt\otimes dx^{i}$, $dx^{i}\otimes dt$ or $dt\otimes dt$. On the other hand, we already know that $A_{v}$ given in formula \eqref{hprojective vector field} satisfies equation \eqref{eq:Av}. After multiplication with $g$ from the left, \eqref{eq:Av} can be written as
\begin{align}
\mathcal{L}_{v}g-\frac{\trace(g^{-1}\circ\mathcal{L}_{v}g)}{2(n+1)}g=a+\beta g\nonumber
\end{align}
for $a=g\circ A$ and some constant $\beta$. Calculating the trace on both sides yields 
\begin{align}
\mathcal{L}_{v}g= a+(\beta +\frac{1}{2}\trace( A+\beta \Id))g= a+((n+1)\beta +\rho+ k_{1}\rho_{1}+ k_{2}\rho_{2})g.\nonumber
\end{align}
Now we can insert equation \eqref{eq:liederiv} on the left-hand side to obtain
\begin{align}
\dot{h}\,dt\otimes dt+\mathcal{L}_{v_{1}}\tilde{g}+\mathcal{L}_{v_{2}}\tilde{g}= a+((n+1)\beta +\rho+ k_{1}\rho_{1}+ k_{2}\rho_{2})g.\label{eq:liederiv2}
\end{align}
Since equation \eqref{eq:liederiv2} is in block-diagonal form, it splits up into two separate equations. The first equation which belongs to the matrix entry on the upper left reads
\begin{align}
\dot{h}=(2 \rho+C)h\mbox{, where we defined }C= k_{1}\rho_{1}+k_{2}\rho_{2}+(n+1)\beta.\nonumber
\end{align}
Integration of this differential equation yields
\begin{align}
h(t)=De^{Ct+2 \int\rho dt}=De^{\left(C- c_{1}\right)t-2\ln(\cosh(\sqrt{\alpha}(t+d)))}\nonumber
\end{align}
for $\alpha=\frac{c_{1}^{2}}{4}-c_{0}>0$ and some constants $d$ and $D>0$. If we insert the formulas $\rho_{1}=-\tfrac{c_{1}}{2}-\sqrt{\alpha}$ and $\rho_{2}=-\tfrac{c_{1}}{2}+\sqrt{\alpha}$ for the constant eigenvalues in the definition of the constant $C$, we obtain 
\begin{align}
C=-\frac{n-1}{2}c_{1}-(2k_{1}+1-n)\sqrt{\alpha}+(n+1)\beta.\nonumber
\end{align}
Finally, Lemma \ref{lem:metricinadaptedcoord} is proven.
\end{proof}
The formulas \eqref{eq:sol} in Lemma \ref{lem:metricinadaptedcoord} show that the restriction
\begin{align}
g_{|E_{A}(\rho)}=\left(\begin{array}{cc}h&0\\0&g(\Lambda,\Lambda)\end{array}\right)\label{eq:metrictwodim}
\end{align}
of the metric to the integral manifolds of the distribution $E_{A}(\rho)=\spann\{v_{1},\bar{\Lambda}\}$ (the coordinates are as in  Lemma \ref{lem:metricinadaptedcoord}, i.e., $\partial_{1}=v_{1}$ and $\partial_{2}=\bar{\Lambda}$) depends on the global constants $c_{1},c_{0},k_{1}$ and $\beta$. The constants $D$ and $d$ are not interesting; they can depend a priori on the particular choice of the coordinate neighborhood.  Note that $c_{1}$ and $c_{0}$ are subject to the condition $\alpha=c_{1}^{2}/4-c_{0}>0$. Now our goal is to show that we can impose further constraints on the constants such that the only metric which is left is the metric of positive constant holomorphic sectional curvature. So far, we did not really use that the manifold is closed, indeed, most of the statements listed above still would be true if this condition is omitted. However,  as the next lemma shows, the condition that $M$ is closed  imposes strong restrictions on the constants from  Lemma \ref{lem:metricinadaptedcoord}:
\begin{lem}
\label{lem:metriccomp}
The constants from the formulas \eqref{eq:sol} of Lemma \ref{lem:metricinadaptedcoord} satisfy $C=c_{1}$. In particular, the function $h=g(v_{1},v_{1})$ has the form 
\begin{align}
h(t)=\frac{D}{\cosh^{2}(\sqrt{\alpha}(t+d))}.\label{eq:sol2}
\end{align}
\end{lem}
\begin{proof}
First we will show that certain integral curves of $v_{1}$ always have finite length.  Let $x_{\mbox{\tiny max}}$ and $x_{\mbox{\tiny min}}$ be points where $\rho$ takes its maximum and minimum values respectively. We consider a geodesic $\gamma:[0,1]\rightarrow M$ joining the points $\gamma(0)=x_{\mbox{\tiny max}}$ and $\gamma(1)=x_{\mbox{\tiny min}}$. We again consider  the integrals $I_{0},I_{1},I_{2}:TM\rightarrow\mathbb{R}$ given by  \eqref{eq:int3}. Since the Killing vector field $\bar{\Lambda}$ vanishes at $x_{\mbox{\tiny max}}$, we obtain that $0=I_{0}(\dot{\gamma}(0))=I_{0}(\dot{\gamma}(t))$ for all $t\in[0,1]$. By Lemma \ref{lem:propv1}, $\rho(x_{\mbox{\tiny max}})$ is equal to the constant eigenvalue $\rho_{2}=-\frac{c_{1}}{2}+\sqrt{\alpha}$. It follows that $I_{2}(\zeta)=0$ for all $\zeta\in T_{x_{\mbox{\tiny max}}}M$, in particular, $I_{2}(\dot{\gamma}(0))=0$. This implies that $I_{2}(\dot{\gamma}(t))=0$ for all $t\in[0,1]$. Similarly, considering the point $x_{\mbox{\tiny min}}$, we obtain $I_{1}(\dot{\gamma}(t))=0$ for all $t\in[0,1]$. In the proof of Lemma \ref{lem:integrable}, we already remarked that the distribution $D_{1}$ is equal to the intersection of the nullities   of $I_{0},I_{1}$ and $I_{2}$. It follows that $\dot{\gamma}(t)$ is contained in $D_{1}$ for all $0<t<1$. This implies that $\gamma_{|(0,1)}$ is a reparametrized integral curve $\sigma:\mathbb{R}\rightarrow M$ of the complete vector field $v_{1}$. In particular, the length 
\begin{align}
l_{g}(\sigma)=\int_{-\infty}^{+\infty}\sqrt{g(\dot{\sigma}(t),\dot{\sigma}(t))}dt=\int_{-\infty}^{+\infty}\sqrt{g(v_{1},v_{1})(\sigma(t))}dt=\int_{-\infty}^{+\infty}\sqrt{h(t)}dt\label{eq:length}
\end{align}
of the curve $\sigma$ is equal to the length $l_{g}(\gamma_{|[0,1]})$ of the geodesic $\gamma$. We obtain that $l_{g}(\sigma)$ is finite. By equation \eqref{eq:length}, a necessary condition for $l_{g}(\sigma)$ to be finite is that $\sqrt{h(t)}\rightarrow 0$ when $t\rightarrow\infty$. Note that $h(t)$ is given by the first equation in \eqref{eq:sol} (for some constants $D,d$ that can depend on the particular integral curve $\sigma$). From formula \eqref{eq:sol}, we obtain that $\sqrt{h(t)}$ for $t\rightarrow\infty$ is asymptotically equal to
\begin{align}
\sqrt{h(t)}\sim e^{\left(\frac{C-c_{1}}{2\sqrt{\alpha}}-1\right)t}.\nonumber
\end{align}
The finiteness of $l_{g}(\sigma)$ now implies the condition
\begin{align}
-\frac{C-c_{1}}{2\sqrt{\alpha}}+1>0\label{eq:restr1}
\end{align}
on the global constants given in equation \eqref{eq:sol}.
Let us find further conditions on the constants. Since $M$ is assumed to be closed, the holomorphic sectional curvature
\begin{align}
K_{E_{A}(\rho)}=\frac{g(v_{1},R(v_{1},\bar{\Lambda})\bar{\Lambda})}{g(v_{1},v_{1})g(\Lambda,\Lambda)}=\frac{R_{1212}}{h\,g(\Lambda,\Lambda)}\nonumber
\end{align}
of $E_{A}(\rho)$ has to be bounded on $M$. Since the integral manifolds of $E_{A}(\rho)$ are totally geodesic (by  Lemma \ref{lem:integrable}), 
the sectional curvature $K_{E_{A}(\rho)}$ is equal to  the curvature of the two dimensional metric \eqref{eq:metrictwodim}. After a straight-forward calculation using the formulas \eqref{eq:sol} for $h$ and $g(\Lambda,\Lambda)$, we obtain 
\begin{align}
K_{E_{A}(\rho)}(t)=&\frac{1}{4D}\Bigg[\underbrace{(-4c_{0}-C^{2}+2Cc_{1})}_{=:\gamma_{1}}e^{-(C-c_{1})t}\underbrace{-(C-c_{1})^{2}}_{=:\gamma_{2}}\,\cosh(2\sqrt{\alpha}(t+d))e^{-(C-c_{1})t}\nonumber\vspace{2mm}\\
&\underbrace{-2(C-c_{1})\sqrt{\alpha}}_{=:\gamma_{3}}\,\mbox{sinh}(2\sqrt{\alpha}(t+d))e^{-(C-c_{1})t}\Bigg]\label{eq:curv4}\\
&=:\frac{1}{4D}(\gamma_{1}f_{1}(t)+\gamma_{2}f_{2}(t)+\gamma_{3}f_{3}(t)).\nonumber
\end{align}
Similar to the first part of the proof, we can consider the asymptotic behavior $t\rightarrow\infty$ of the functions $f_{2}(t),f_{3}(t)$ appearing as coefficients of the constants $\gamma_{2},\gamma_{3}$ in formula \eqref{eq:curv4}. We substitute $s=2\sqrt{\alpha}(t+d)$ and obtain
\begin{align}
\begin{array}{l}
f_{2}(s)\sim\cosh(s)e^{-\frac{C-c_{1}}{2\sqrt{\alpha}}s}\underset{t\gg0}{\sim} e^{\left(-\frac{C-c_{1}}{2\sqrt{\alpha}}+1\right)t}\overset{\eqref{eq:restr1}}{\underset{t\rightarrow\infty}{\longrightarrow}}\infty,\vspace{2mm}\\
f_{3}(s)\sim\mbox{sinh}(s)e^{-\frac{C-c_{1}}{2\sqrt{\alpha}}s}\underset{t\gg0}{\sim} e^{\left(-\frac{C-c_{1}}{2\sqrt{\alpha}}+1\right)t}\overset{\eqref{eq:restr1}}{\underset{t\rightarrow\infty}{\longrightarrow}}\infty.
\end{array}\nonumber
\end{align}
As we already have mentioned, the sectional curvatures of a closed manifold are bounded and hence, $K_{E_{A}(\rho)}(t)$ must be finite when $t$ approaches the limit $t\rightarrow\infty$. Using the formulas for the asymptotic behavior of $f_{2}(t)$ and $f_{3}(t)$ given above, this condition imposes the restriction $\gamma_{2}=-\gamma_{3}$ on the constants in equation \eqref{eq:curv4}.
Similarly, considering the asymptotic behaviour for $t\rightarrow-\infty$, we obtain $\gamma_{2}=\gamma_{3}$. Note that the dominating part in $\mbox{sinh}(2\sqrt{\alpha}(t+d))$ now comes with the minus sign. It follows that $\gamma_{2}=\gamma_{3}=0$, hence, 
\begin{align}
C-c_{1}=0\label{eq:restr2}
\end{align}
as we claimed. Inserting equation \eqref{eq:restr2} in the first formula of \eqref{eq:sol}, the metric component $g_{11}=h$ takes the form \eqref{eq:sol2}.
Lemma \ref{lem:metriccomp} is proven.
\end{proof}
\begin{rem}
If we insert $\gamma_{2}=\gamma_{3}=0$ and $C=c_{1}$ in the formula \eqref{eq:curv4} for the sectional curvature of $E_{A}(\rho)$, we obtain that $K_{E_{A}(\rho)}=\tfrac{\alpha}{D}$ is constant and positive as we claimed.
\end{rem}

\subsection{Proof of Theorem \ref{thm:system}}
\label{sec:systemloc}
Our goal is to prove  Theorem \ref{thm:system}:  we need  to show the local existence of a function $\mu$ and a constant $B$ such that the system \eqref{eq:system} is satisfied.

\begin{lem}
\label{lem:covderiv}
At every point $x\in M$, the  tensor $A$ and the covariant differential $\nabla\Lambda$ are simultaneously diagonalizable in an orthogonal basis. More precisely, let $U\in E_{A}(\rho_{1})$ and $W\in E_{A}(\rho_{2})$ be eigenvectors of $A$ corresponding to the constant eigenvalues. Then we obtain
\begin{align}
\begin{array}{l}
\nabla_{\Lambda}\Lambda=(\dot{\phi}+\phi\psi)\Lambda,\vspace{2mm}\\
\nabla_{\bar{\Lambda}}\Lambda=(\dot{\phi}+\phi\psi)\bar{\Lambda},\vspace{2mm}\\
\nabla_{U}\Lambda=\frac{g(\Lambda,\Lambda)}{\rho-\rho_{1}}U,\vspace{2mm}\\
\nabla_{W}\Lambda=\frac{g(\Lambda,\Lambda)}{\rho-\rho_{2}}W.
\end{array}\label{eq:covderiv2}
\end{align}
The functions $\phi$ and $\psi$ are given by the formulas
\begin{align}
\phi=\frac{1}{2}\frac{\dot{\rho}}{h}\mbox{ and }\psi=\frac{1}{2}\frac{\dot{h}}{h}.\label{eq:mn}
\end{align}
\end{lem}
\begin{proof}
Since the distribution $D_{1}$ has totally geodesic integral manifolds (see Lemma \ref{lem:integrable}), $\nabla_{v_{1}}v_{1}$ is proportional to $v_{1}$. Let us define two functions $\phi$ and $\psi$ by setting
\begin{align}
\Lambda=:\phi v_{1}\mbox{ and }\nabla_{v_{1}}v_{1}=:\psi v_{1}.\label{eq:coeff}
\end{align}
It follows immediately that $g(\Lambda,\Lambda)=\phi^{2}h$. On the other hand, $\dot{h}=2g(\nabla_{v_{1}}v_{1},v_{1})=2\psi h$. Using the equations \eqref{eq:sol} in Lemma \ref{lem:metricinadaptedcoord}, we obtain
\begin{align}
\phi=\frac{1}{2}\frac{\dot{\rho}}{h}\mbox{ and }\psi=\frac{1}{2}\frac{\dot{h}}{h}.\label{eq:coeff2}
\end{align}
Note that the function $\phi$ has to be negative since $\rho$ decreases along the flow-lines of $v_{1}$ while it increases along the flow-lines of $\Lambda=\frac{1}{2}\grad\,\rho$. By direct  calculation, we obtain 
\begin{align}
\nabla_{\Lambda}\Lambda=\phi\nabla_{v_{1}}(\phi v_{1})=\phi \dot{\phi}v_{1}+\phi^{2}\nabla_{v_{1}}v_{1}=(\phi\dot{\phi}+\phi^{2}\psi)v_{1}=(\dot{\phi}+\phi\psi)\Lambda.\nonumber
\end{align}
From the equation above, the relation $\bar{\Lambda}=J\Lambda$ and the fact that $\Lambda$ is a holomorphic vector field, we immediately obtain
\begin{align}
\nabla_{\bar{\Lambda}}\Lambda=J\nabla_{\Lambda}\Lambda=(\dot{\phi}+\phi\psi)\bar{\Lambda}\nonumber
\end{align}
and hence, the first two equations in \eqref{eq:covderiv2} are proven.\\
Now let $U\in E_{A}(\rho_{1})$ be an eigenvector of $A$ corresponding to the constant eigenvalue $\rho_{1}$.
Using Proposition \ref{prop:modulo}, we obtain
\begin{align}
\nabla_{U}\bar{\Lambda}=-\frac{g(\Lambda,\Lambda)}{\rho_{1}-\rho}JU+f\bar{\Lambda}+\tilde{f}\Lambda\mbox{ and }\nabla_{\bar{\Lambda}}U=0\mbox{ mod }E_{A}(\rho_{1})\label{eq:covderiv}
\end{align}
for some functions $f$ and $\tilde{f}$. The lie bracket of $U$ and $\bar{\Lambda}$ is given by
\begin{align}
[U,\bar{\Lambda}]=f\bar{\Lambda}+\tilde{f}\Lambda\mbox{ mod }E_{A}(\rho_{1}).\nonumber
\end{align}
Applying $d\rho$ to both sides of the equation above yields $\tilde{f}\Lambda(\rho)=0$. Since $\Lambda(\rho)\neq 0$ on $M_{d\rho\neq0}$, it follows that $\tilde{f}=0$. On the other hand, the first equation in \eqref{eq:covderiv} shows that
\begin{align}
\frac{1}{2}U(g(\Lambda,\Lambda))=g(\nabla_{U}\bar{\Lambda},\bar{\Lambda})=fg(\Lambda,\Lambda).\nonumber
\end{align}
Since $dg(\Lambda,\Lambda)$ is zero when restricted to the distribution $D_{2}$ (as can be seen by using the coordinates given in Lemma \ref{lem:metricinadaptedcoord}), the left-hand side of the equation above vanishes and hence, $f=0$. Inserting $f=\tilde{f}=0$ in the first equation of \eqref{eq:covderiv}, we obtain the third equation in \eqref{eq:covderiv2}. If we replace $\rho_{1}$ and $U$ by $\rho_{2}$ and $W\in E_{A}(\rho_{2})$, the same arguments can be applied to obtain the last equation in \eqref{eq:covderiv2}. Lemma \ref{lem:covderiv} is proven.
\end{proof}
Let $(M,g,J)$ be a closed, connected Riemannian K\"ahler manifold of real dimension $2n\geq 4$ and of degree of mobility $D(g)=2$. Let $v$ be an essential $h$-projective vector field and $t_{0}$ a real number, such that $\bar{g}=(\Phi^{v}_{t_{0}})^{*}g$ is not already affinely equivalent to $g$. Let us denote by $A=A(g,\bar{g})$ the corresponding solution of equation (\ref{f}).\\
We want to show that  any point of $M_{d\rho\neq0}$  has a small neighborhood such that 
in this neighborhood  there exist  a function $\mu$ and a constant $B<0$  such that the covariant differential $\nabla\Lambda$ satisfies the second equation 
\begin{align}
\nabla\Lambda=\mu \Id+B A\label{eq:system2}
\end{align}
in \eqref{eq:system}.
By  Lemma \ref{lem:covderiv}, at every point of  $M_{d\rho\neq0}$, each eigenvector of $A$ is an eigenvector of $\nabla\Lambda$. Since $A$ has (at most) three different eigenvalues, equation \eqref{eq:system2} is equivalent to an inhomogeneous linear system of three equations on the two unknown real numbers $\mu$ and $B$. Using formulas  \eqref{eq:covderiv2} from Lemma \ref{lem:covderiv},  we see that for $x\in M_{d\rho\neq0}$, $\nabla\Lambda$ satisfies equation \eqref{eq:system2} for some numbers $\mu$ and $B$, if and only if the inhomogeneous linear system of equations
\begin{align}
\begin{array}{l}
\mu+\rho B=\dot{\phi}+\phi\psi,\vspace{2mm}\\
\mu+\rho_1 B=\frac{g(\Lambda,\Lambda)}{\rho-\rho_1},\vspace{2mm}\\
\mu+\rho_2 B=\frac{g(\Lambda,\Lambda)}{\rho-\rho_2}.\vspace{2mm}
\end{array}\label{eq:linsys}
\end{align}
is satisfied. Now,  according to Lemma \ref{lem:metricinadaptedcoord} and Lemma \ref{lem:metriccomp}, in a neighborhood of a point of $M_{d\rho\neq0}$, the functions $\rho$, $g(\Lambda,\Lambda)$, $h$, $\phi$ and $\psi$ are given explicitly by  \eqref{eq:sol}, \eqref{eq:sol2} and \eqref{eq:mn}. Let us insert these functions and the formulas $-\frac{c_{1}}{2}\pm\sqrt{\alpha}$ for the constant eigenvalues $\rho_{1}<\rho_{2}$ (see Lemma \ref{lem:spectrumD=2}) in \eqref{eq:linsys}. After a straight-forward calculation, we obtain  that  \eqref{eq:linsys} is satisfied for  
\begin{align}
\mu=-\frac{\alpha(\frac{c_{1}}{2}-\sqrt{\alpha}\,\mbox{tanh}(\sqrt{\alpha}(t+d)))}{4D}=B(c_{1}+\rho)\mbox{ and }B=-\frac{\alpha}{4D}.\label{eq:solextsys}
\end{align}
We see also  that the constant $B$ is  negative  (as we claimed in Section \ref{sec:scheme}).

Using the equation $\lambda=\frac{1}{4}\trace\,A=\frac{1}{2}\rho+\mbox{const}$, we obtain that $\mu$ given by \eqref{eq:solextsys} satisfies $d\mu=Bd\rho=2Bd\lambda$. Since $\Lambda$ is the gradient of $\lambda$, this is easily seen to be equivalent to the third equation in the system \eqref{eq:system}.

We have shown that in a neighborhood of almost every point of $M$, there exists  a  smooth 
function $\mu$ and a constant $B<0$, such that the system \eqref{eq:system} is satisfied for the triple $(A,\Lambda,\mu)$.

If $\tilde{A}$ is another element in $\mbox{Sol}(g)$ with the corresponding vector field $\tilde{\Lambda}$, then $\tilde{A}=aA+b\Id$ for some $a,b\in \mathbb{R}$ implying $ \tilde \Lambda= a \Lambda$. By direct calculations we see that for an appropriate local function $\tilde{\mu}$ the triple  $(\tilde{A},\tilde{\Lambda},\tilde{\mu})$ satisfies the system \eqref{eq:system} for  the same constant $\tilde{B}=B$. Finally, Theorem \ref{thm:system} is proven.

\section{Final step in the proof of Theorem \ref{thm:obata}}
\label{sec:obata}
As we explained in Section \ref{sec:scheme}, it is sufficient to prove Theorem \ref{thm:obata} under the additional assumption that the  degree of mobility is equal to two. By  Theorem \ref{thm:system}, for every   $A\in \mbox{Sol}(g)$ with corresponding vector field $\Lambda=\tfrac{1}{4}\grad\,\trace\,A$, in a neighborhood $U(x)$ of almost every point $x\in M$, there exists  a local function $\mu:U(x)\rightarrow \mathbb{R}$ and a negative constant $B$   such that the triple $(A,\Lambda,\mu)$ satisfies the system \eqref{eq:system}.

Now, 
in   \cite[\S2.5]{FKMR} it was shown that under these assumptions  the constant $B$  is the same for all  such neighborhoods,    implying that the system  \eqref{eq:system}  is satisfied on the whole $M$ (for a certain smooth function 
  $\mu:M\to \mathbb{R}$). Note that in view of the third equation of \eqref{eq:system}, $\mu$ is not a constant (if $A$ is chosen to be non-proportional to the identity on $TM$).

By direct calculation (differentiating $\mu$ and replacing  the derivatives using the system \eqref{eq:system}), we 
 obtain
\begin{align}
(\nabla\nabla\mu)(Y,Z)&=\nabla_{Y}(\nabla_{Z}\mu)-\nabla_{\nabla_{Y}Z}\mu\overset{\mbox{\tiny eq. $3$ of }\eqref{eq:system}}{=}2Bg(Z,\nabla_{Y}\Lambda)\nonumber\\&\overset{\mbox{\tiny eq. $2$ of }\eqref{eq:system}}{=}2B(\mu g(Y,Z)+Bg(AY,Z)).\nonumber
\end{align}
Then, 
\begin{align}
\begin{array}{c}
(\nabla\nabla\nabla\mu)(X,Y,Z)=2B((\nabla_{X}\mu)g(Y,Z)+Bg((\nabla_{X}A)Y,Z))\vspace{2mm}\\
\overset{\mbox{\tiny eq. $1$ of }\eqref{eq:system}}{=}B(2(\nabla_{X}\mu)g(Y,Z)+2Bg(Z,\Lambda)g(X,Y)+2Bg(Y,\Lambda)g(X,Z)\vspace{2mm}\\+2Bg(Z,\bar{\Lambda})g(JX,Y)+2Bg(Y,\bar{\Lambda})g(JX,Z)).\nonumber
\end{array}
\end{align}
Inserting the third equation of \eqref{eq:system}, we obtain  that $\mu$ satisfies the equation
\begin{align}
\begin{array}{c}
(\nabla\nabla\nabla\mu)(X,Y,Z)=B[2(\nabla_{X}\mu)g(Y,Z)+(\nabla_{Z}\mu)g(X,Y)+(\nabla_{Y}\mu)g(X,Z)\vspace{2mm}\\-(\nabla_{JZ}\mu)g(JX,Y)-(\nabla_{JY}\mu)g(JX,Z)]
\end{array}
\label{eq:tanno}
\end{align}
for all $X,Y,Z\in TM$.\\
Now  by   \cite[Theorem 10.1]{Tanno1978}, the existence of a non-constant solution of equation \eqref{eq:tanno} with $B<0$ on a closed, connected Riemannian K\"ahler manifold implies that the manifold has positive constant holomorphic sectional curvature equal to $-4B$. Then,  $(M,-4Bg,J)$ can be covered by $(\mathbb{C}P(n),g_{FS},J)$.  Theorem \ref{thm:obata} is proven.

\subsection*{\bf Acknowledgements.} We thank N. Hitchin for noticing  that $h$-projectively equivalent metrics and  hamiltonian $2$-forms are essentially the same object. This work benefited from discussions with  D. Calderbank and 
 C. Tønnesen-Friedman. We  thank    Deutsche Forschungsgemeinschaft (Research training group   1523 --- Quantum and Gravitational Fields)   and FSU Jena for partial financial support.

\begin{appendix}
\section{$H$-projectively invariant formulation of the main equation \eqref{f}}
\label{app}
\subsection{$H$-projective structure}
Let $(M,J)$ be a complex manifold of real dimension $2n$. Note that the defining equation \eqref{eq:eq1} for $h$-planar curves only involves the connection - it does not depend on the metric.
\begin{defn}
Two symmetric complex (i.e., $DJ=\bar{D}J=0$) affine connections $D$ and $\bar{D}$ are called \emph{$h$-projectively equivalent} if each $h$-planar curve with respect to $D$ is $h$-planar with respect to $\bar{D}$ and vice versa.
\end{defn}
It is a classical result (see for example \cite{Otsuki1954,Tashiro1956}) that two symmetric complex affine connections $D$ and $\bar{D}$ are $h$-projectively equivalent if and only if for a certain $1$-form $\Phi$ we have
\begin{align}
\bar{D}_{X}Y-D_{X}Y=\Phi(Y)X+\Phi(X)Y-\Phi(JY)JX-\Phi(JX)JY\label{eq:hproconn}
\end{align}
for all vector fields $X,Y$.
\begin{rem}
\label{rem:kaehlerconn}
If the symmetric affine connections $D$ and $\bar{D}$ are related by \eqref{eq:hproconn} and $DJ=0$, then $\bar{D}J=0$ as well.
\end{rem}
\begin{defn}
An \emph{$h$-projective structure} on $(M,J)$ is an equivalence class $[D]$ of $h$-projectively equivalent symmetric complex affine connections.
\end{defn}
\subsection{$H$-projectively invariant version of equation \eqref{f}}
Let $(M,J)$ be a complex manifold of real dimension $2n$. Denote by $\wedge^{2n}:=\wedge^{2n}T^{*}M$ the bundle of $2n$-forms on $M$. Note that it is a trivial line bundle since the complex manifold $(M,J)$ is always orientable. The bundle $(\wedge^{2n})^{\frac{w}{2(n+1)}}$ of $2n$-forms of "$h$-projective weight" $w$ is an one-dimensional bundle whose   transition functions  are the transition functions of 
 $\wedge^{2n}$ (which can be chosen to have positive values) to  the power $\tfrac{w}{2(n+1)}$. Let us consider the bundle $S^{2}_{J}TM$ of symmetric hermitian (with respect to $J$) $(2,0)$-tensors and define its "weighted" version $S_{J}^{2}TM(w)$  by 
$$S_{J}^{2}TM(w):=S_{J}^{2}TM\otimes(\wedge^{2n})^{\frac{w}{2(n+1)}}.$$ 
For each choice of local coordinates $x^{1},...,x^{2n}$, the local section $dx^{1}\wedge...\wedge dx^{2n}$ of $\wedge^{2n}$ gives us a trivialization for $(\wedge^{2n})^{\frac{w}{2(n+1)}}$. Then, we can think that  a section $\sigma\in \Gamma(S_{J}^{2}TM(w))$ is  a symmetric hermitian $2n\times 2n-$matrix with components $\sigma^{ij}=\sigma^{ij}(x^{1},...,x^{2n})$.  If we make an orientation-preserving change of coordinates $x^{1},...,x^{2n}\longmapsto \tilde{x}^{1},...,\tilde{x}^{2n}$, the components $\sigma^{ij}$ transform according to the rule 
\begin{align}
\tilde{\sigma}^{ij}=\left(\mathrm{det}\,\left(\frac{\partial\,\tilde{x}^{k}}{\partial\,x^{l}}\right)\right)^{-\frac{w}{2(n+1)}}\frac{\partial\, \tilde{x}^{i}}{\partial\,x^{p}}\frac{\partial\, \tilde{x}^{j}}{\partial\,x^{q}}\sigma^{pq}.\label{eq:trafolaw3}
\end{align}
The covariant derivative of  $\sigma\in\Gamma(S_{J}^{2}TM(w))$ with respect to a symmetric  affine connection $D$ is  given by
\begin{align}
D_{k}\sigma^{ij}=\underbrace{\partial_{k}\sigma^{ij}+\Gamma^{i}_{kl}\sigma^{lj}+\Gamma^{j}_{kl}\sigma^{il}}_{\mbox{\tiny usual covariant derivative for $2$-tensors}}\underbrace{-\frac{w}{2(n+1)}\Gamma^{l}_{kl}\sigma^{ij}}_{\begin{array}{c}\mbox{\tiny addition corresponding to}\\\mbox{\tiny $2n$-forms of weight $w$}\end{array}},\label{eq:inducedconn}
\end{align}
where $\Gamma^{i}_{jk}$ are the Christoffel symbols of $D$.
\begin{thm}
\label{thm:invariant}
Let $\sigma$ be an element of $\Gamma(S_{J}^{2}TM(2))$. Consider the equation 
\begin{align}
D_{k}\sigma^{ij}-\frac{1}{2n}(\delta^{i}_{k}D_{l}\sigma^{l j}+\delta^{j}_{k}D_{l}\sigma^{l i}+J^{i}_{k}J^{j}_{m}D_{l}\sigma^{lm}+J^{j}_{k}J^{i}_{m}D_{l}\sigma^{lm})=0.\label{eq:weightedmain}
\end{align}
Then, the following holds:
\begin{enumerate}
\item Equation \eqref{eq:weightedmain} is $h$-projectively invariant, i.e., independent of the connection $D\in[D]$.
\item Equation \eqref{eq:weightedmain} has a non-degenerate solution $\sigma$ (where non-degeneracy means that the matrix of components $(\sigma^{ij})$ is invertible everywhere), if and only if there is a connection $\nabla\in[D]$, such that $\nabla$ is the Levi-Civita connection of some Kähler metric.
\end{enumerate}
\end{thm}
\begin{rem}
We do not pretend that Theorem \ref{thm:invariant} is new;  it was known to D. Calderbank (private communication), and is  completely  analogous to the corresponding statement for  projective structures   treated in \cite{EastMat}.
\end{rem}
\begin{rem}  For  a nondegenerate  solution $\sigma$ of  \eqref{eq:weightedmain}, 
the metric $g$ given by formula \eqref{petrunin} below is   a Kähler metric  that is $h$-projectively equivalent to the connection $D$.  
\end{rem}

\begin{proof}
$(1)$ The condition \eqref{eq:hproconn} for the $h$-projective equivalence of the connections $D$ and $\bar{D}$ can be rewritten locally as
\begin{align}
\bar{\Gamma}^{i}_{jk}-\Gamma^{i}_{jk}=\delta^{i}_{j}\Phi_{k}+\delta^{i}_{k}\Phi_{j}-J^{i}_{j}J^{l}_{k}\Phi_{l}-J^{i}_{k}J^{l}_{j}\Phi_{l},\label{eq:hproconncoord}
\end{align}
where $\Gamma^{i}_{jk}$ and $\bar{\Gamma}^{i}_{jk}$ are the Christoffel symbols of $D$ and $\bar{D}$ respectively. 
Combining the equations \eqref{eq:hproconncoord} and \eqref{eq:inducedconn}, we can calculate the difference between the connections $D$ and $\bar{D}$ when they are acting on $\sigma\in\Gamma(S^{2}_{J}TM(2))$. We obtain 
\begin{align}
\bar{D}_{k}\sigma^{ij}=D_{k}\sigma^{ij}+\delta^{i}_{k}\Phi_{l}\sigma^{l j}+\delta^{j}_{k}\Phi_{l}\sigma^{il}+J^{i}_{k}J^{j}_{m}\Phi_{l}\sigma^{lm}+J^{j}_{k}J^{i}_{m}\Phi_{l}\sigma^{lm},\label{eq:trafolaw1}
\end{align}  
and in particular,
\begin{align}
\bar{D}_{l}\sigma^{l j}=D_{l}\sigma^{l j}+2n\Phi_{l}\sigma^{l j}.\label{eq:trafolaw2}
\end{align}
Replacing $D$ with $\bar{D}$ in equation \eqref{eq:weightedmain} and inserting the transformation laws \eqref{eq:trafolaw1} and \eqref{eq:trafolaw2}, we obtain that \eqref{eq:weightedmain} remains unchanged if $D$ is replaced by $\bar{D}\in[D]$.

$(2)$ In one direction, $(2)$ is trivial. Suppose that $g$ is a Kähler metric that is $h$-projectively equivalent to $D$. Let us denote by $g^{-1}\in\Gamma(S^{2}_{J}TM)$ the dual of $g$ (i.e., $g\circ g^{-1}=Id$). We consider the non-degenerate element
$$\sigma=g^{-1}\otimes(\mathrm{vol}_{g})^{\frac{1}{n+1}}\in \Gamma({S^{2}_{J}TM(2)}).$$
Evidently, $\nabla\sigma=0$, where $\nabla$ is the Levi-Civita connection of $g$. By the first part of Theorem \ref{thm:invariant}, $D$ can be replaced with $\nabla$ in \eqref{eq:weightedmain} and we obtain that $\sigma$ is a solution of \eqref{eq:weightedmain}.

Let us proof $(2)$ in the opposite direction. Let $\sigma\in\Gamma(S^{2}_{J}TM(2))$ be a non-degenerate solution of \eqref{eq:weightedmain}. Using the transformation law \eqref{eq:trafolaw3}, it is easy to see that 
\begin{equation} \label{petrunin} g^{ij}=\sigma^{ij}|\mathrm{det}\,(\sigma^{ij})|^{\frac{1}{2}}\end{equation}
defines the components of a symmetric, hermitian (with respect to $J$) $(2,0)$-tensor. Thus the corresponding dual $(0,2)$-tensor $g$ is a hermitian metric. Note that $\sigma$ and $g$ are related by
$$\sigma=g^{-1}\otimes(\mathrm{vol}_{g})^{\frac{1}{n+1}}.$$
It remains to show that the Levi-Civita connection of $g$ is contained in $[D]$.  We consider a connection $\bar{D}\in[D]$ related to $D$ by \eqref{eq:hproconncoord} such that
\begin{align}
\Phi_{i}=-\tfrac{1}{2n}\sigma_{im}D_{l}\sigma^{lm},\label{eq:specialphi} 
\end{align}
where the components $\sigma_{ij}$ are defined by $\sigma^{i p }\sigma_{p j}=\delta^{i}_{j}$. Substituting \eqref{eq:specialphi} in equation \eqref{eq:trafolaw2} shows that
\begin{align}
\bar{D}_{l}\sigma^{lj}=0.\label{eq:vanishingtrace}
\end{align}
Replacing $D$ with $\bar{D}$ in \eqref{eq:weightedmain} and substituting \eqref{eq:vanishingtrace}, we obtain $\bar{D}_{k}\sigma^{ij}=0$. Thus, $\bar{D}$ is the Levi-Civita connection of $g$. By Remark \ref{rem:kaehlerconn}, $\bar{D}$ satisfies $\bar{D}J=0$ which implies that $g$ is indeed a Kähler metric.
\end{proof}

\begin{defn}
Let $[D]$ be an $h$-projective structure on the complex manifold $(M,J)$. We denote by $\mbox{Sol}([D])\subseteq\Gamma(S^{2}_{J}TM(2))$ the linear space of solutions of equation \eqref{eq:weightedmain}.
\end{defn}

\subsection{An alternative proof of Lemma \ref{lem:liederivA}}
Let $(M,g,J)$ be a Kähler manifold and let $\nabla$ be the Levi-Civita connection of $g$. We assume that the degree of mobility (see Definition \ref{def:deg}) is equal to two. Clearly, we have that $\mbox{Sol}([\nabla])$ is $2$-dimensional. Suppose that the Kähler metric $\bar{g}$ is non-proportional and $h$-projectively equivalent to $g$ and consider the corresponding elements
\begin{align}
\sigma=g^{-1}\otimes(\mathrm{vol}_{g})^{\frac{1}{n+1}}\mbox{ and }\bar{\sigma}=\bar{g}^{-1}\otimes(\mathrm{vol}_{\bar{g}})^{\frac{1}{n+1}},\label{eq:basisforsol}
\end{align}
of $\mbox{Sol}([\nabla])$. Since $g$ and $\bar{g}$ are non-proportional, $\sigma$ and $\bar{\sigma}$ form a basis for $\mbox{Sol}([\nabla])$.

Now let $v$ be an $h$-projective vector field for $(M,g,J)$ (see Definition \ref{def:hprotrafo}). Thus, the Lie derivative $\mathcal{L}_{v}$ maps solutions of \eqref{eq:weightedmain} to solutions of \eqref{eq:weightedmain} and, hence, restricts to an endomorphism of the $2$-dimensional vector space $\mbox{Sol}([\nabla])$. With respect to the basis $\sigma,\bar{\sigma}$ of $\mbox{Sol}([\nabla])$, the endomorphism $\mathcal{L}_{v}$ is given by
\begin{align}
\begin{array}{c}
\mathcal{L}_{v}\sigma=\kappa_{11}\sigma+\kappa_{12}\bar{\sigma},\\
\mathcal{L}_{v}\bar{\sigma}=\kappa_{21}\sigma+\kappa_{22}\bar{\sigma}.\end{array}\label{eq:matrix}
\end{align}
for some real numbers $\kappa_{11},\kappa_{12},\kappa_{21},\kappa_{22}$.

Consider the $(1,1)$-tensor $A:=\bar{\sigma}\sigma^{-1}$. Combining \eqref{eq:basisforsol} with equation \eqref{eq:a}, we see that $A$ coincides with $A(g,\bar{g})$. We calculate
$$\mathcal{L}_{v}A=(\mathcal{L}_{v}\bar{\sigma})\sigma^{-1}+\bar{\sigma}(\mathcal{L}_{v}\sigma^{-1})=(\mathcal{L}_{v}\bar{\sigma})\sigma^{-1}-\bar{\sigma}\sigma^{-1}(\mathcal{L}_{v}\sigma)\sigma^{-1}.$$
Substituting \eqref{eq:matrix}, we obtain
\begin{align}
\mathcal{L}_{v}A&=(\kappa_{21}\sigma+\kappa_{22}\bar{\sigma})\sigma^{-1}-\bar{\sigma}\sigma^{-1}(\kappa_{11}\sigma+\kappa_{12}\bar{\sigma})\sigma^{-1}\nonumber\\
&=\kappa_{21}Id+(\kappa_{22}-\kappa_{11})A-\kappa_{12}A^{2}.\nonumber
\end{align}
Hence, 
$$\mathcal{L}_{v}A=c_{2}A^{2}+c_{1}A+c_{0}Id$$
for some constants $c_{2},c_{1},c_{0}$. This is the assertion of Lemma \ref{lem:liederivA}. 
\end{appendix}

\nocite{*}
\bibliographystyle{plain}

\end{document}